\documentclass[12pt]{article}

\usepackage{amsmath}
\usepackage{latexsym}
\usepackage{amsfonts}
\usepackage{amssymb}
\usepackage[russian,english]{babel}

\usepackage{latexsym}
\newcounter{theorem}

\renewcommand{\thetheorem}{\arabic{section}.\arabic{theorem}}

\newenvironment{thm}[1]{\par\addvspace{0.5cm}
    \begin{sloppypar}\refstepcounter{theorem}%
    {\bf #1 \thetheorem.}\it{}}{\end{sloppypar}}

\newcommand{\eh}{\hfill}\newlength{\sperr}

\newenvironment{theorem}{\begin{thm}{Theorem}} {\end{thm}}

\newenvironment{corollary}{\begin{thm}{Corollary}} {\end{thm}}

\newenvironment{defi}[1]{\par\addvspace{0.5cm}
\begin{sloppypar}\refstepcounter{theorem}%
{\bf #1 \thetheorem.}\rm{}}{\end{sloppypar}}

\newenvironment{definition}{\begin{defi}{Definition}}{\end{defi}}

\newenvironment{remark}{\begin{defi}{Remark}}{\end{defi}}

\newenvironment{proof}{{\settowidth{\sperr}{\rm Proof}
\par\addvspace{0.3cm}\parbox[t]{1.3\sperr}{\rm P\eh r\eh o\eh o\eh f\eh. }%
}}{\nopagebreak\mbox{}\hfill $\Box$\par\addvspace{0.25cm}}

\newcommand{\essi}{\operatornamewithlimits{ess\,inf}}
\newcommand{\esss}{\operatornamewithlimits{ess\,sup}}

\newcommand{\al}{\alpha}
\newcommand{\bt}{\beta}
\newcommand{\dl}{\delta}

\newcommand{\Om}{\Omega}
\newcommand{\lb}{\lambda}

\newcommand{\ve}{\varepsilon}
\newcommand{\gm}{\gamma}
\newcommand{\Gm}{\Gamma}
\newcommand{\sg}{\sigma}

\newcommand{\vi}{\varphi}

\newcommand{\rn}{\mathbb{R}^n}
\newcommand{\rone}{\mathbb{R}^1}

\newcommand{\intl}{\int\limits}

\newcommand{\diam}{\mbox{\,\rm diam\,}}

\begin{document}

\centerline{\textbf{Operators of Harmonic Analysis }}

\centerline{\textbf{ in Weighted  Spaces with Non-standard Growth}}

 \vspace{4mm}

\centerline{by}

 \centerline{\textbf{V.M. Kokilashvili,}}

 \centerline{\textit{A.Razmadze Mathematical Institute and Black
Sea University, Tbilisi, Georgia}}

 \centerline{\textit{kokil@rmi.acnet.ge}}

\vspace{5mm}
 \centerline{and }

 \vspace{5mm}
\centerline{\textbf{S.G.Samko}}

\centerline{\textit{Universidade do Algarve, Portugal}}

\centerline{\textit{ssamko@ualg.pt}}

 \vspace{5mm}
\centerline{\textbf{Abstracts}}

 \vspace{3mm}

 \footnotesize

Last years there was increasing an interest to the so called function spaces with non-standard growth, known also
as variable exponent Lebesgue spaces. For weighted such spaces on homogeneous spaces, we develop a certain
variant of Rubio de Francia's extrapolation theorem. This extrapolation theorem is applied to obtain the
boundedness in such spaces of various operators of harmonic analysis, such as
maximal and singular operators,
potential operators, Fourier multipliers, dominants of partial sums of trigonometric Fourier series and others,
in weighted Lebesgue spaces with variable exponent. There are also given their vector-valued analogues.

 \vspace{3mm}
\normalsize

\section{Introduction}\label{1}

\setcounter{equation}{0}

During  last years a significant progress was made in the study of maximal  and singular operators
and potential type operators in the generalized Lebesgue spaces $L^{p(\cdot)}$ with variable
exponent, known also as the spaces with non-standard growth. A number of mathematical problems
leading to such spaces with variable exponent arise in applications to partial differential
equations, variational problems and continuum mechanics (in particular, in the theory of the so
called electrorheological fluids), see E. Acerbi and G.Mingione \cite{9b},\cite{9d}, X.Fan and
D.Zhao \cite{160zb}, M.Ru$\check{z}$i$\check{c}$ka \cite{525}, V.V. Zhikov \cite{730ab},
\cite{730c}.
 These applications stipulated a significant interest to such spaces in the last decade.

  The most advance in the study of the classical operators of harmonic analysis in the case of variable
  exponent was made in the
Euclidean setting, including weighted estimates.
 We refer in particular to the surveying articles  L.Diening, P.H{\"a}st{\"o} and A.Nekvinda
\cite{106b},  V.Kokilashvili \cite{316b}, S.Samko \cite{580bd} and
papers D.Cruz-Uribe, A.Fiorenza, J.M.Martell and C.Perez
\cite{101zb}, D.Cruz-Uribe, A.Fiorenza and C.J.Neugebauer
 \cite{101ab}, L. Diening \cite{106}, \cite{105a}, \cite{106z}, L.Diening and M.Ru$\check{z}$i$\check{c}$ka \cite{107a},
V. Kokilashvili, N.Samko and S.Samko \cite{317c}, V.Kokilashvili and S.Samko \cite{321j},  \cite{321c},
\cite{321a}, \cite{321i}, A.Nekvinda  \cite{414b}, S.Samko
 \cite{579}, \cite{580b},  \cite{580bc},  S.Samko, E.Shargorodsky and B.Vakulov \cite{584a} and references therein.

Recently there also started the investigation of these classical
operators in the spaces with variable exponent in the setting of
metric measure spaces,  the case of constant $p$ in this setting
having a long history, we refer, in particular to the papers
A.P.Calder{\'o}n \cite{72b}, R.R.Coifman and G.Weiss \cite{97},
\cite{97a},  R.Mac\'{\i}as and C.Segovia \cite{381a}, books
D.E.Edmunds and V.Kokilashvili and A.Meskhi \cite{145a} and
I.Genebashvili, A.Gogatishvili, V.Kokilashvili and M.Krbec
\cite{187}, J.Heinonen \cite{225a} and references therein. The
non-weighted boundedness of the maximal operator on homogeneous
spaces was proved by P.Harjulehto, P.H{\"a}st{\"o} and M.Pere
\cite{224b} and Sobolev embedding theorem with variable exponents
on homogeneous spaces with variable dimension was proved in
P.Harjulehto, P.H{\"a}st{\"o} and V.Latvala \cite{224ab}.

In the present paper we give  a development of weighted
estimations  of various operators of harmonic analysis in Lebesgue
spaces with variable exponent $p(x)$. We first give theorems on
the weighted boundedness of the maximal operator on homogeneous
spaces (Theorems \ref{th3.1.} and \ref{th3.2.}). Next, in Section
\ref{subs4.} we give a certain $p(\cdot)\to q(\cdot)$-version of
Rubio de Francia's extrapolation theorem \cite{522b} within the
frameworks of weighted spaces $L_\varrho^{p(\cdot)}$ on metric
measure spaces. Proving this version we develop some ideas and
approaches of papers \cite{101zb}, \cite{101ac}.

By means of this extrapolation theorem and known theorems on the
boundedness with Muckenhoupt weights in the case of constant $p$,
we obtain results on  weighted $p(\cdot)\to q(\cdot)$- or $p(\cdot)\to p(\cdot)$-boundedness - in the case of variable exponent $p(x)$ -  of the following operators \\
1) potential type operators, \\
2)  Fourier multipliers (weighted Mikhlin, H\"ormander and Lizorkin-type theorems, Subsection \ref{subs5.1.}),\\
3) multipliers of trigonometric Fourier series (Subsection
\ref{subs5.2.}),\\
3) majorants of partial sums of Fourier series (Subsection
\ref{subs5.3.}),\\
4) singular integral operators on Carleson curves and in Euclidean setting (Subsections \ref{subs5.4.}-\ref{subs5.8.}),\\
5) Feffermann-Stein function (Subsection \ref{subs5.7.}),\\
6) some vector-valued operators (Subsection \ref{subs5.9.}).

 \section{Definitions and preliminaries}\label{sec1}

\setcounter{equation}{0} \setcounter{theorem}{0}

\subsection{On variable dimensions in metric measure
spaces}\label{subskiki}

In the sequel, $(X,d,\mu)$ denotes a metric space  with the
(quasi)metric
 $d$  and non-negative measure  $\mu$. We refer  to  \cite{145a}, \cite{187}, \cite{225a} for the basics on metric measure spaces. By $B(x,r)=\{y\in X: d(x,y)<r\}$ we denote a
ball in $X$. The following standard conditions will be assumed to
be satisfied: \\
1) all the balls $B(x,r)=\{y\in X: d(x,y)<r\}$ are measurable,
\\
2) the space $C(X)$ of uniformly continuous functions on $X$ is
dense in $L^1(\mu)$.

In most of the statements we also suppose that \\
 3) the measure
$\mu$ satisfies the doubling condition:
$$\mu B(x,2r)\le C \mu B(x,r),$$
where $C>0$ does not depend on  $r>0$ and $x\in X.$ A measure
satisfying this condition will be called doubling measure.

For a locally $\mu$-integrable function  $f: X \to \mathbb{R}^1$ we consider the Hardy-Littlewood  maximal
function
$$
\mathcal{M} f(x)=\sup_{r>0} \frac{1}{\mu (B(x,r))}
\int\limits_{B(x,r)} |f(y)| \,d\mu(y).
$$
By  $A_s=A_s(X)$, where $1\le s <\infty$,  we denote the class of
weights  (locally almost everywhere positive $\mu$-integrable
functions)  $w: X\to \mathbb{R}^1$ which satisfy  the Muckenhoupt
condition
$$\sup\limits_{B}\left(\frac{1}{\mu
B}\intl_{B}w(y)d\mu(y)\right)\left(\frac{1}{\mu B}\intl_{B}
w^{-\frac{1}{s-1}}(y)d\mu(y)\right)^{s-1} <\infty$$ in the case
$1<s<\infty$, and the condition
$$\mathcal{M}w(x)\le Cw(x) $$
for almost all  $x\in X$, with a constant  $C>0$, not depending on
$x\in X$, in the case  $s=1$. Obviously,  $A_1\subset A_s, \
1<s<\infty.$

As is known, see   \cite{72b}, \cite{381a},  the weighted boundedness
$$\intl\limits_{X}(\mathcal{M}f(x))^s w(x) d\mu(x) \le C \intl\limits_{X}|f(x)|^s w(x) d\mu(x),$$
holds, if and only if  $w\in A_s$.

Let  $\Om$ be an open set in  $X$.

\begin{definition}\label{def1.1chch}
By  $\mathcal{P}(\Om)$ we denote the class of $\mu$-measurable
functions on $\Om$, such that
\begin{equation}\label{1.1}
1<p_-\le p_+<\infty,
\end{equation}
where  $ p_-=p_-(\Om)=\essi\limits_{x\in\Om} p(x) \quad \textrm{and} \quad p_+=p_+(\Om)=\esss\limits_{x\in\Om}
p(x). $
\end{definition}

\begin{definition}\label{def1.1ch}
By $L_\varrho^{p(\cdot)}(\Om)$ we denote the weighted Banach
function space  of $\mu$-measurable functions  $f: \Om\to
\mathbb{R}_1^+$, such that
\begin{equation}\label{1.2}
\|f\|_{L^{p(\cdot)}_\varrho}:=\|\varrho
f\|_{p(\cdot)}=\inf\left\{\lb>0: \intl_\Om
\left|\frac{\varrho(x)f(x)}{\lb}\right|^{p(x)} \;d\mu(x)\le
1\right\}<\infty .
\end{equation}
\end{definition}

\vspace{4mm} \begin{definition}\label{def1.1} \textit{We say that
a weight  $\varrho$ belongs to the class
$\mathfrak{A}_{p(\cdot)}(\Om)$, if the maximal operator $\mathcal{M}$
is bounded in the space $L_\varrho^{p(\cdot)}(\Om).$}
\end{definition}
\begin{definition}\label{def1.1}
\textit{A function $p:\Om \to \mathbb{R}^1$ is said to belong to
the class $WL(\Om)$ (weak Lipshitz), if}
\begin{equation}\label{1.3}
|p(x)-p(y)|\leq \frac{A}{\ln\frac{1}{d(x,y)}}\,, \;\; d(x,y)\leq
\frac{1}{2}, \;\; x,\,y\in \Om,
\end{equation}
\textit{where $A>0$ does not depend on  $x$ and $y$.}
\end{definition}

\vspace{3mm}  The notion of lower and upper local dimension  of
$X$ at a point $x$ introduced as
$$\underline{dim}\,X(x)= \lim\limits_{\overline{r\to 0}} \frac{\ln \mu B(x,r)}{\ln r}, \quad
\overline{dim}\, X(x)= \overline{\lim\limits_{r\to 0}}\, \frac{\ln \mu B(x,r)}{\ln r}$$ are known,
see e.g. \cite{160zzzz}. We will use different notions of local lower and upper dimensions,
inspired by the notion of the so called index numbers $m(w), M(w)$ of almost monotonic functions
$w$, see their definition in (\ref{mÌ}). These indices  studied in \cite{539}, \cite{539d},
\cite{539e}, are versions of Matuzewska-Orlicz index numbers used in the theory of Orlicz spaces,
see \cite{382a}, \cite{382b}. The idea to introduce local dimensions in terms of these indices by
the following definition was borrowed from the papers \cite{539j}, \cite{539jnew}.

\begin{definition}\label{defN}
 The numbers
\begin{equation}\label{ibas}
\underline{\mathfrak{dim}}(X;x) =\sup_{r>1}\frac{\ln   \
\left(\lim\limits_{\overline{h\to 0}} \frac{\mu B(x,rh)}{\mu
B(x,h)} \right)}{\ln   \ r}\  , \quad
\overline{\mathfrak{dim}}(X;x) =\inf_{r>1}\frac{\ln \
\left(\overline{\lim\limits_{h\to 0}}\frac{\mu B(x,rh)}{\mu
B(x,h)}
  \right)}{\ln   \ r}
\end{equation}
will be referred to as local lower and upper dimensions.
\end{definition}
  Observe
that the "dimension" $\underline{\mathfrak{dim}}(X;x)$ may be also
rewritten in terms of the upper limit as well:
\begin{equation} \label{ggbvcdshnew}
 \underline{\mathfrak{dim}}(X;x) =\sup_{0<r<1}\frac{\ln   \ \left(\overline{\lim\limits_{h\to 0}}
\frac{\mu B(x,rh)}{\mu B (x,h)} \right)}{\ln   \ r}.
\end{equation}
Since the function
\begin{equation}\label{dsgh}
\mu_0(x,r) = \overline{\lim\limits_{h\to 0}} \frac{\mu
B(x,rh)}{\mu B(x,h)}
\end{equation}
is semimultiplicative in $r$, that is, $\mu_0(x,r_1r_2)\le
\mu_0(x,r_1)\mu_0(x,r_2)$, by  properties  of such functions
(\cite{342}, p. 75; \cite{342a}) we obtain that
$\underline{\mathfrak{dim}}(X;x)\le
\overline{\mathfrak{dim}}(X;x)$ and we may  rewrite the dimensions
$\underline{\mathfrak{dim}}(X;x)$ and
$\overline{\mathfrak{dim}}(X;x)$ also in the form
\begin{equation}\label{ksajlkJ}
\underline{\mathfrak{dim}}(X;x) = \lim\limits_{r\to
 0}\frac{\ln   \mu_0(x,r)}{\ln   \ r}, \quad \overline{\mathfrak{dim}}(X;x) = \lim\limits_{r\to
\infty}\frac{\ln   \mu_0(x,r)}{\ln   \ r}.
\end{equation}

\begin{remark}\label{rteyriu}
Introduction of dimensions $\underline{\mathfrak{dim}}(X;x)$ and
$\overline{\mathfrak{dim}}(X;x)$ just in form
(\ref{ggbvcdshnew})-(\ref{ksajlkJ})  is caused by the fact that
they arise naturally when dealing with Muckenhoupt condition for
radial type weights on metric measure spaces. They seem may not
coincide with dimensions $\underline{dim}\,X(x), \overline{dim}\,
X(x)$. There is an impression that probably for different goals
different notions of dimensions may be useful.
\end{remark}

 \vspace{2mm}

We will mainly need the lower bound for lower dimensions
$\underline{\mathfrak{dim}}(X;x)$ on an open set $\Om\subseteq X$:
$$\underline{\mathfrak{dim}}(\Om):=\essi\limits_{x\in X}\underline{\mathfrak{dim}}(\Om;x).$$

In case where $\Om$ is unbounded, we will also need similar
dimensions connected in a sense with the influence of  infinity.
Let
\begin{equation}\label{dsghbuyt}
\mu_\infty(x,r) = \overline{\lim\limits_{h\to \infty}} \frac{\mu
B(x,rh)}{\mu B(x,h)}.
\end{equation}
We introduce the numbers
\begin{equation}\label{ksajlkJmnxt}
\underline{\mathfrak{dim}}_\infty(X;x) = \lim\limits_{r\to
 0}\frac{\ln   \mu_\infty(x,r)}{\ln   \ r}, \quad \overline{\mathfrak{dim}}_\infty(X;x) =
  \lim\limits_{r\to \infty}\frac{\ln   \mu_\infty(x,r)}{\ln   \ r}
\end{equation}
and their bounds
\begin{equation}\label{i}
\underline{\mathfrak{dim}}_\infty(\Om)= \essi\limits_{x\in \Om} \underline{\mathfrak{dim}}_\infty (X;x), \quad
\overline{\mathfrak{dim}}_\infty(\Om)= \esss\limits_{x\in \Om} \overline{\mathfrak{dim}}_\infty(X;x).
\end{equation}

It is not hard to see that  $ \underline{\mathfrak{dim}}(\Om),
\underline{\mathfrak{dim}}_\infty(\Om),$ and $
\overline{\mathfrak{dim}}_\infty(\Om)$ are non-negative. In the
sequel, when considering these bounds of dimensions we always
assume that $\underline{\mathfrak{dim}}(\Om),  $ $
\underline{\mathfrak{dim}}_\infty(\Om), \
\overline{\mathfrak{dim}}_\infty (\Om) \in (0,\infty)$.

\vspace{5mm} \subsection{Classes of the weight
functions}\label{subsec2.}

We consider, in particular, the weights
\begin{equation}\label{2.1}
\varrho(x)= [1+d(x_0,x)]^{\bt_\infty}\prod\limits_{k=1}^N
[d(x,x_k)]^{\bt_k}, \ \ \  x_k\in X, k=0,1,...N,
\end{equation}
where  $\bt_\infty=0$ â in the case where  $X$ is bounded. Let
$\Pi=\{x_0,x_1,..., x_N\}$ be a  given finite set of points in
$X$. We take $d(x,y)=|x-y|$ in all the cases where
$X=\mathbb{R}^n$.

\vspace{4mm} \begin{definition}\label{deff2.1.} \textit{A weight
function of form (\ref{2.1}) is said to belong to the class
$V_{p(\cdot)}(\Om,\Pi)$, where $p(\cdot)\in C(\Om)$, if}
\begin{equation}\label{2.2}
-\frac{\underline{\mathfrak{dim}}(\Om)}{p(x_k)}<\bt_k<\frac{\underline{\mathfrak{dim}}(\Om)}{p^\prime(x_k)}
\end{equation}
 \textit{and, in the case $\Om$ is infinite,}
\begin{equation}\label{2.3}
  -\frac{\underline{\mathfrak{dim}}_\infty(\Om)}{p_\infty}<\bt_\infty+
  \sum\limits_{k=1}^N\bt_k <\underline{\mathfrak{dim}}_\infty(\Om)
  -\frac{\overline{\mathfrak{dim}}_\infty(\Om)}{p_\infty}.
\end{equation}
\end{definition}

Note that when the metric space  $X$ has a constant dimension $s$
in the sense that
$$c_1 r^s\le \mu B(x,r)\le c_2 r^s$$
with the constants  $c_1>0$ è $c_2>0$, not depending on
 $x\in X$ and $r>0$, the inequalities in   (\ref{2.2}), (\ref{2.3})
 and (\ref{2.6}) turn into
\begin{equation}\label{2.2prime}
-\frac{s}{p(x_k)}<\bt_k<\frac{s}{p^\prime(x_k)},
 \quad  -\frac{s}{p_\infty}<\bt_\infty+\sum\limits_{k=1}^N\bt_k
<\frac{s}{p^\prime_\infty}
\end{equation}
and
\begin{equation}\label{2.6prime}
 -\frac{s}{p(x_k)}
< m(w)\le M(w) < \frac{s}{p^\prime(x_k)}\, \ , \ \ k=1,2,...,N,
\end{equation}
respectively.

In fact, we may admit  a more general class of weights
\begin{equation}\label{2.4}
\varrho(x)=w_0[1+d(x_0,x)]\prod_{k=1}^N w_k[d(x,x_k)]
\end{equation}
with  "radial"  weights, where  the functions $w_0$ and $w_k,
k=1,...,N,$ belong to a class of Zygmund-Bary-Stechkin type, which
admits an oscillation between two power functions with different
exponents.

  \vspace{4mm}
By  $U=U([0,\ell])$ we denote the class of functions $u\in C([0,\ell]), \ 0<\ell\le \infty,$ such
that  $ u(0)=0, u(t)>0$ for $t>0$ and  $u$ is an almost increasing function on $[0,\ell]$. (We
recall that a function $u$ is called \textit{almost increasing} on $[0,\ell]$, if there exists a
constant $C (\ge 1)$ such that $u(t_1)\le u(t_2)$ for all $0\le t_1\le t_2 \le \ell$). By
$\widetilde{U}$ we denote the class of function $u$, such that
 $t^au(t)\in U$ for some  $a\in\mathbb{R}^1$.

  \vspace{4mm}
  \begin{definition}\label{def2.2.} \ (\cite{46}) \ \textit{A function $v$
is said to belong to the Zygmund-Bary-Stechkin class $\Phi^0_\dl$, if
$$
 \int_0^h\frac{v(t)}{t}dt
\le cv(h) \ \ \textit{and} \ \
 \int_h^\ell\frac{v(t)}{t^{1+\dl}}dt \le c\frac{v(h)}{h^\dl},
$$
where  $c=c(v)>0$ does not depend on   $h\in (0,\ell]$.}
\end{definition}

It is known that $v\in \Phi^0_\dl$, if and only if $0<m(v)\le
M(v)<\dl$, where
\begin{equation}\label{mÌ}
m(w)=\sup_{t>1}\frac{\ln\left(\lim\limits_{\overline{h\to 0}} \frac{w(ht)}{w(h)}\right)}{\ln t} \ \ \ \
\textrm{and} \ \ \ \ M(w)=\sup_{t>1}\frac{\ln\left(\overline{\lim\limits_{h\to 0}} \frac{w(ht)}{w(h)}\right)}{\ln
t}
\end{equation}
 (see \cite{539}, \cite{539d}, \cite{270a}).

For functions $ w$ defined in the neighborhood of infinity and
such that $w\left(\frac{1}{r}\right)\in \widetilde{U}([0, \dl]) $
 for some $\dl >0$,
 we introduce also
\begin{equation}\label{msusutxe}
 m_\infty(w) =\sup_{x>1}\frac{\ln \
\left[\underline{\lim}_{h\to \infty} \frac{w(xh)}{w(h)}
\right]}{\ln \ x}\ , \ \  M_\infty(w) =\inf_{x>1}\frac{\ln \
\left[\overline{\lim}_{h\to \infty} \frac{w(xh)}{w(h)}
\right]}{\ln \ x}.
\end{equation}

Generalizing Definition \ref{deff2.1.}, we introduce also the
following notion.
  \begin{definition}\label{def2.3.}
\textit{A weight function $\varrho$ of form (\ref{2.4}) is said to belong to the class
$V^{osc}_{p(\cdot)}(\Om,\Pi)$, where $p(\cdot)\in C(\Om)$, if
\begin{equation}\label{2.6}
w_k(r)\in \widetilde{U}([0,\ell]), \ell=\diam \Om \quad
\textrm{and}\ \  -\frac{\underline{\mathfrak{dim}}(\Om)}{p(x_k)} <
m(w_k)\le M(w_k) <
\frac{\underline{\mathfrak{dim}}(\Om)}{p^\prime(x_k)} ,
\end{equation}
$ k=1,2,...,N,$
and (in the case $\Om$ is infinite)
$$w_0\left(\frac{\ell^2}{r}\right)\in \widetilde{U}([0,\ell]) $$
and
\begin{equation}\label{f27dco}
 -\frac{\underline{\mathfrak{dim}}_\infty(\Om)}{p_\infty}<\sum\limits_{k=0}^N m_\infty(w_k)\le
\sum\limits_{k=0}^N M_\infty(w_k)
<\frac{\underline{\mathfrak{dim}}_\infty(\Om)}{p^\prime_\infty}-\Delta_{p_\infty},
\end{equation}
where
$\Delta_{p_\infty}=\frac{\overline{\mathfrak{dim}}_\infty(\Om)-\underline{\mathfrak{dim}}_\infty(\Om)}{p_\infty}.$}
\end{definition}

Observe that in the case $\Om=X=\mathbb{R}^n$ conditions
(\ref{2.6}) and (\ref{f27dco}) take the form
\begin{equation}\label{2.6bvc}
w_k(r)\in \widetilde{U}(\mathbf{R}^1_+): =\left\{w: \ w\left(r\right), w\left(\frac{1}{r}\right)\in
\widetilde{U}([0,1])\right\}
\end{equation}
 and
\begin{equation}\label{f27dcobvc}
 -\frac{n}{p(x_k)} < m(w_k)\le M(w_k) < \frac{n}{p^\prime(x_k)}
,\quad  -\frac{n}{p_\infty}<\sum\limits_{k=0}^N m_\infty(w_k)\le \sum\limits_{k=0}^N M_\infty(w_k)
<\frac{n}{p^\prime_\infty}.
\end{equation}

 \begin{remark}\label{rem4.3.} \textit{For every $p_0\in (1,p_-)$  there hold the implications}
  $$\varrho\in V_{p(\cdot)}(\Om,\Pi) \ \Longrightarrow \ \varrho^{-p_0}\in
  V_{(\widetilde{p})^\prime(\cdot)}(\Om,\Pi)$$
\textit{and}
$$\varrho\in V^{osc}_{p(\cdot)}(\Om,\Pi) \ \Longrightarrow \ \varrho^{-p_0}\in
   V^{osc}_{(\widetilde{p})^\prime(\cdot)}(\Om,\Pi),$$
\textit{where $\widetilde{p}(x)=\frac{p(x)}{p_0}$.}
\end{remark}

\vspace{5mm} \subsection{The boundedness of the Hardy-Littlewood
maximal operator on metric spaces with doubling measure, in
weighted Lebesgue spaces with variable exponent}\label{subsec3.}

The following statements are valid.

 \begin{theorem} \label{th3.1.} \textit{ Let $X$ be a metric space with doubling
 measure and
  let $\Om$ be bounded. If  $p\in \mathcal{P}(\Om)\cap WL(\Om)$ and $\varrho\in
V_{p(\cdot)}^{osc}(\Om,\Pi)$, then $\mathcal{M}$ is bounded in the
space  $L_\varrho^{p(\cdot)}(\Om)$.}
\end{theorem}

\begin{theorem} \label{th3.2.}  \textit{ Let $X$ be a metric space with doubling
 measure and
  let $\Om$ be
unbounded. Let $p\in \mathcal{P}(\Om)\cap WL(\Om)$ and  let there
exist $R>0$ such that $p(x)\equiv p_\infty=const $ for $x\in \Om
\backslash{B(x_0,R)}$. If $\varrho\in
V^{osc}_{p(\cdot)}(\Om,\Pi)$, then $\mathcal{M}$ is bounded in the
space
 $L_\varrho^{p(\cdot)}(\Om)$.}
\end{theorem}

\vspace{3mm} The Euclidean version of Theorems \ref{th3.1.} and \ref{th3.2.}  was proved in  \cite{106} in the
non-weighted  case and in \cite{317c}, \cite{JFSA} in the weighted case; in \cite{JFSA} there were also proved
the corresponding  versions of Theorems \ref{th3.1.} and \ref{th3.2.} for the maximal operator on Carleson curves
(a typical example of metric measure spaces with constant dimension). The proof of Theorems \ref{th3.1.} and
\ref{th3.2.} in the general case in main is similar, being based on the approaches used in the proofs for the
case of Carleson curves.

\begin{theorem} \label{erz}  Let $\Om$ be a bounded open set in a doubling measure metric space $X$, let
the exponent $p(x)$ satisfy conditions (\ref{1.1}), (\ref{1.3}).
Then the operator $\mathcal{M}$ is bounded in
$L^{p(\cdot)}_\varrho(\Om)$, if $$[\varrho(x)]^{p(x)}\in
A_{p_-}(\Om).$$
\end{theorem}

\vspace{3mm}We refer to \cite{newmetric} for Theorem \ref{th3.2.},
its detailed proof for the case where $X$ is a Carleson curve  is
given in \cite{JFSA}, the proof for a doubling measure metric
space being in fact the same.

\vspace{5mm} \section{Extrapolation theorem on metric measure
spaces}\label{subs4.}

\setcounter{equation}{0} \setcounter{theorem}{0}

In the sequel $\mathcal{F}=\mathcal{F}(\Om)$ denotes a family  of
ordered pairs
 $(f,g)$ of non-negative $\mu$-measurable functions  $f,g$,
 defined  on an open set   $\Om \subset X$. When saying that  there holds
 an
 inequality of  type (\ref{new1})
for all pairs  $(f,g)\in \mathcal{F}$ and weights $w\in A_1$, we always mean that it is valid for all the pairs,
for which the left-hand side is finite, and that the constant $c$ depends only on $p_0,q_0$ and the
$A_1$-constant of the weight.

In what follows, by $p_0$ and $q_0$ we denote  positive  numbers
such that
\begin{equation}\label{vgschi}
0<p_0\le q_0<\infty, \quad p_0<p_- \quad \textrm{and} \quad
\frac{1}{p_0}- \frac{1}{p_+}<\frac{1}{q_0}
\end{equation}
and use the notation
\begin{equation}\label{nfse3a}
\widetilde{p}(x)=\frac{p(x)}{p_0},  \quad
\widetilde{q}(x)=\frac{q(x)}{q_0}.
\end{equation}
\begin{remark}\label{nond} The extrapolation Theorem \ref{th4.1.}  with variable exponents in the
non-weighted case $\varrho(x)\equiv 1$ and in the Euclidean setting was proved in \cite{101zb}. For extrapolation
theorems in the case of constant exponents we refer to \cite{522b}, \cite{HMS}.

Observe that  the measure $\mu$ in Theorem \ref{th4.1.} is not assumed to be doubling.
\end{remark}

 \begin{theorem}\label{th4.1.}
Let  $X$ be a metric measure space and  $\Om$  an open set in $X$. Assume that for some $p_0$ and $q_0$,
satisfying conditions (\ref{vgschi}) and every weight $w\in A_1(\Om)$ there holds the inequality
\begin{equation}\label{new1}
\left(\intl\limits_{\Om}f^{q_0}(x)w(x)d\mu(x)\right)^\frac{1}{q_0}\le
c_0
\left(\intl\limits_{\Om}g^{p_0}(x)[w(x)]^\frac{p_0}{q_0}d\mu(x)\right)^\frac{1}{p_0}
\end{equation}
for all $f,g$ in a given family $\mathcal{F}$. Let the variable
exponent $q(x)$ be defined by
\begin{equation}\label{s9x34}
\frac{1}{q(x)}=
\frac{1}{p(x)}-\left(\frac{1}{p_0}-\frac{1}{q_0}\right),
\end{equation}
let  the exponent
 $p(x)$ and the weight $\varrho(x)$ satisfy the conditions
\begin{equation}\label{new1bcxc54esa}
p \in \mathcal{P}(\Om) \quad \textrm{and}\ \quad
\varrho^{-q_0} \in
 \mathfrak{A}_{(\widetilde{q})^\prime}(\Om).
\end{equation}
Then for all $(f,g)\in\mathcal{F}$ with $f\in
L_\varrho^{p(\cdot)}(\Om)$ the inequality
\begin{equation}\label{new2}
 \|f\|_{L_\varrho^{q(\cdot)}} \le C
\|g\|_{L_\varrho^{p(\cdot)}}
\end{equation}
 is valid with a constant  $C>0$, not depending on $f$ and $g$.
\end{theorem}

\begin{proof} By the Riesz theorem, valid for the spaces with variable
exponent in the case $1<p_-\le p_+<\infty$, (see \cite{332},
\cite{575a}), we have
$$\|f\|_{L_\varrho^{q(\cdot)}}^{q_0}= \|f^{q_0}\varrho^{q_0}\|_{L^{\widetilde{q}(\cdot)}}\le
\sup \intl_\Om f^{p_0}(x)h(x)d\mu(x),  $$ where we assume that $f$
is non-negative and $\sup $ is taken with respect to all
non-negative $h$ such that
$\|h\varrho^{-q_0}\|_{L^{(\widetilde{q})^\prime(\cdot)}}\le 1$. We
fix any such a function $h$. Let us show  that
\begin{equation}\label{4.2}
\intl\limits_{\Om}f^{q_0}(x)h(x)d\mu(x)\le C
\|g\varrho\|^{q_0}_{L^{q(\cdot)}}
\end{equation}
for an arbitrary pair   $(f,g)$  from the given family
$\mathcal{F}$ with a constant  $C>0$, not depending on  $h, f$ and
$g$. By the assumption  $\varrho^{-q_0} \in
 \mathfrak{A}_{(\widetilde{q})^\prime}(\Om)$   we have
\begin{equation}\label{4.3}
\|\varrho^{-q_0}\mathcal{M}\vi\|_{L^{\widetilde{q}^\prime(\cdot)}(\Om)}\le
C_0 \|\varrho^{-q_0}\vi\|_{L^{\widetilde{p}^\prime(\cdot)}(\Om)}
\end{equation}
where the constant  $C_0>0$ does not depend on $\vi$.

We make use of the following construction which is due to Rubio de
Francia \cite{522b}
\begin{equation}\label{4.4}
S\vi(x)=\sum\limits_{k=0}^\infty (2C_0)^{-k}\mathcal{M}^k\vi(x),
\end{equation}
where $\mathcal{M}^k$ is the  $k$-iterated maximal operator and
$C_0$ is the constant from  (\ref{4.3}) (one may take  $C_0\ge
1$). The following statements are obvious:

1) \ $\vi(x) \le S\vi (x), \ \  x\in\Om $ \ for any non-negative
function $\vi$;
\begin{equation}\label{4.5}
\ \ \ \  \ \ 2) \ \ \ \hspace{24mm}
\|\varrho^{-q_0}S\vi\|_{L^{(\widetilde{q})^\prime}(\Om)} \le 2
\|\varrho^{-q_0}\vi\|_{L^{(\widetilde{q})^\prime}(\Om)},
\hspace{36mm}
\end{equation}

3) \ \ $\mathcal{M}(S\vi)(x)\le 2C_0 S\vi(x), \ \ \  \ x\in\Om,$

\noindent so that  $S\vi\in A_1(\Om)$ with the $A_1$-constant not
depending on  $\vi$.  Therefore $S\vi\in A_{q_0}(\Om)$.

By 1), for $\vi=h$ we have
\begin{equation}\label{4.6}
\intl_\Om f^{q_0}(x)h(x)d\mu(x) \le  \intl_\Om
f^{q_0}(x)Sh(x)d\mu(x).
\end{equation}
By the H\"older inequality for variable exponent, property 2) and
the condition $f\in L_\varrho^{q(\cdot)}$, we have
$$\intl_\Om
f^{q_0}(x)Sh(x)d\mu(x) \le k
\|f^{q_0}\varrho^{q_0}\|_{L^{\widetilde{q}(\cdot)}} \cdot
\|\varrho^{-q_0}Sh\|_{L^{(\widetilde{q})^\prime(\cdot)}}$$
$$\le C \|f\varrho\|^{q_0}_{L^{q(\cdot)}} \cdot
\|h\varrho^{-q_0}\|_{L^{(\widetilde{q})^\prime(\cdot)}} \le C \|f\varrho\|_{L^{q(\cdot)}}^{q_0}
<\infty .$$ Consequently,  the integral
 $\intl_\Om f^{q_0}(x)Sh(x)d\mu(x)$ is finite, which enables us  to
 make use of condition  (\ref{new1}) with respect to
 the right-hand side of  (\ref{4.6}). Condition (\ref{new1}) being assumed to be valid with an
  arbitrary weight $w\in A_1$, is in particular valid for $w=Sh$. Therefore,
$$\intl_\Om f^{q_0}(x)Sh(x)d\mu(x)\le C \left(\intl_\Om g^{p_0}(x)[Sh(x)]^\frac{p_0}{q_0}
d\mu(x)\right)^\frac{q_0}{p_0}.$$ Applying the H\"older inequality
on the right-hand side, we get
$$
 \intl_\Om
f^{q_0}(x)Sh(x)d\mu(x)\le C \left(\|g^{p_0}\varrho^{p_0}\|_{L^\frac{p(\cdot)}{p_0}}
\left\|(Sh)^\frac{p_0}{q_0}\varrho^{-p_0}\right\|_{L^{(\widetilde{p})^\prime}}\right)^\frac{q_0}{p_0}.
$$
Thus
\begin{equation}\label{4.7}
\intl_\Om f^{q_0}(x)Sh(x)d\mu(x)\le C \left\|\varrho g\right\|_{L^{p(\cdot)}}^{q_0}
\left\|\varrho^{-p_0}
(Sh)^\frac{p_0}{q_0}\right\|^\frac{q_0}{p_0}_{L^{(\widetilde{p})^\prime(\cdot)}}.
\end{equation}

From (\ref{s9x34}) we easily obtain that
$(\widetilde{p})^\prime(x)=\frac{q_0}{p_0}(\widetilde{q})^\prime(x)$ and then
$$\left\|\varrho^{-p_0}
(Sh)^\frac{p_0}{q_0}\right\|^\frac{q_0}{p_0}_{L^({\widetilde{p})^\prime(\cdot)}}=
\left\|\varrho^{-q_0}
Sh\right\|_{L^{\widetilde{q}^\prime}(\cdot)}.$$ Consequently,
\begin{equation}\label{4f7}
\intl_\Om f^{q_0}(x)Sh(x)d\mu(x)\le C \left\|\varrho
g\right\|_{L^{p(\cdot)}}^{q_0} \left\|\varrho^{-q_0}
Sh\right\|_{L^{\widetilde{q}^\prime}(\cdot)}.
\end{equation}
 To prove (\ref{4.2}), in view
of (\ref{4f7}) it suffices to show that $\left\|\varrho^{-q_0}
Sh\right\|_{L^{\widetilde{q}^\prime}(\cdot)}$ may be estimated by
a constant not depending on  $h$. This follows from (\ref{4.5})
and the condition
$\|h\varrho^{-q_0}\|_{L^{(\widetilde{q})^\prime(\cdot)}}\le 1$ and
proves the theorem.
\end{proof}

\begin{remark}\label{rem1}
It is easy to check that in view of Theorem \ref{erz} the condition
\begin{equation}\label{nasiu6f}
[\varrho(y)]^{q_1(y)} \in A_s, \quad \textrm{where} \quad q_1(y)= \frac{q(y)(q_+-q_0)}{q(y)-q_0} \ \textrm{and} \
s= \frac{q_+}{q_0},
\end{equation}
is sufficient for the validity of the condition $\varrho^{-q_0} \in
 \mathfrak{A}_{(\widetilde{q})^\prime}(\Om)$ of Theorem \ref{th4.1.}.

\end{remark}

By means of Theorems \ref{th3.1.} and \ref{th3.2.}, we obtain the following statement as an immediate consequence
of Theorem \ref{th4.1.} in which we denote
$$\gm= \frac{1}{p_0}-\frac{1}{q_0}.$$

 \begin{theorem}\label{th4.2.} Let $X$ be a metric space
 with doubling
measure  and  $\Om$  an open set in $X$. Let also the following be satisfied\\
 1) \ $p\in
\mathcal{P}(\Om)\cap WL(\Om)$, and in the case $\Om$ is an unbounded set, let $p(x)\equiv
p_\infty=const$ for $x\in
\Om\backslash B(x_0,R)$ with some $x_0\in \Om$ and $R>0$;\\
2) \ there holds inequality (\ref{new1}) for some $p_0$ and $q_0$
satisfying the assumptions in (\ref{vgschi}) and   all
$(f,g)\in\mathcal{F}$ from some family $\mathcal{F}$ and every
weight  $w\in A_1(\Om)$.
  Then\\
  I) \  there holds
inequality (\ref{new2})  for all pairs $(f,g)$ from the same
family $\mathcal{F}$ , such that   $f\in
L_\varrho^{p(\cdot)}(\Om)$ and weights   $\varrho$ of form
(\ref{2.4}) where
\begin{equation}\label{jui7}
\left(\gm-\frac{1}{p(x_k)}\right)\underline{\mathfrak{dim}}(\Om)
<m(w_k)\le M(w_k)<
\left(\frac{1}{p^\prime(x_k)}-\frac{1}{p^\prime_0}\right)\underline{\mathfrak{dim}}(\Om)
\end{equation}
and, in case $\Om$ is unbounded,
\begin{equation}\label{jun70}
\dl +
\left(\gm-\frac{1}{p_\infty}\right)\underline{\mathfrak{dim}}(\Om)<\sum\limits_{k=0}^N
m(w_k)\le \sum\limits_{k=0}^N M(w_k)<
\left(\frac{1}{p^\prime_\infty}-\frac{1}{p^\prime_0}\right)\underline{\mathfrak{dim}}(\Om),
\end{equation}
where
$$\dl = \left[\overline{\mathfrak{dim}}_\infty(\Om)-\underline{\mathfrak{dim}}_\infty(\Om)\right]
\left(\frac{1}{p_0}- \frac{1}{p_\infty}\right);$$
\\
II) \  in case  inequality (\ref{new1}) holds for all $p_0\in
(1,p_-)$, the term $\frac{1}{p^\prime_0}$ in (\ref{jui7}) and
(\ref{jun70}) may be omitted and $\dl$ may be taken in the form
$\dl
=\left[\overline{\mathfrak{dim}}_\infty(\Om)-\underline{\mathfrak{dim}}_\infty(\Om)\right]
\left(\frac{1}{p_-}- \frac{1}{p_\infty}\right).$
\end{theorem}

\section{Application to  problems of the  boundedness
in $ L_\varrho^{p(\cdot)}$ of classical operators of harmonic
analysis }\label{subs5.}

\setcounter{equation}{0} \setcounter{theorem}{0}

\subsection{Potentials operators and fractional maximal function}\label{subspotetials}

We first apply Theorem \ref{th4.1.} to potential operators
\begin{equation}\label{dochnet}
I^\gm_X f(x)= \intl_X \frac{f(y)\,d\mu(y)}{\mu B (x,d(x,y))^{1-\gm}}
\end{equation}
where $0<\gm<1$. We assume that $\mu X=\infty$ and the measure $\mu$ satisfies the doubling condition. We also
additionally suppose the following conditions to be fulfilled:
\begin{equation}\label{condit1}
\textrm{there exists a point} \ \ x_0\in X \ \ \textrm{such that}\ \ \  \mu(x_0)=0
\end{equation}
and
\begin{equation}\label{condit2}
\mu(B(x_0,R)\backslash B(x_0,r))>0 \ \ \ \textrm{for all}   \ \ \ \ 0<r<R<\infty.
\end{equation}

The following statement is valid, see for instance \cite{145a}, p. 412.
\begin{theorem}\label{thKokil}
Let $X$ be a metric measure space with doubling measure satisfying conditions
(\ref{condit1})-(\ref{condit2}), $\mu X =\infty$, let $0<\gm<1$, $1<p_0<\frac{1}{\gm}$ and $
\frac{1}{q_0}=\frac{1}{p_0}-\gm$. The operator $I^\gm_X$ admits the estimate
\begin{equation}\label{admits}
\left(\intl_X |v(x)I_X^\gm f(x)|^{q_0}d\mu\right)^\frac{1}{q_0}\le \left(\intl_X |
v(x)f(x)|^{p_0}d\mu\right)^\frac{1}{p_0},
\end{equation}
if the weight $v(x)$ satisfies the condition
\begin{equation}\label{muck-wheed}
\sup\limits_B\left(\frac{1}{\mu B}\intl_Bv^{q_0}(x)d\mu\right)^\frac{1}{q_0} \left(\frac{1}{\mu
B}\intl_B v^{-p^\prime_0}(x)d\mu\right)^\frac{1}{p^\prime_0}<\infty
\end{equation}
where $B$ stands for a ball in $X$.
\end{theorem}

By means of Theorem \ref{thKokil} and extrapolation Theorem \ref{th4.1.} we arrive at the following statement.

\begin{theorem}\label{thpoten} Let $X$ satisfy the assumptions of Theorem \ref{thKokil}, let $p\in \mathcal{P}$,
$0<\gm<1$ and $p_+<\frac{1}{\gm}$. The weighted  estimate
\begin{equation}\label{Sobolev1}
\left\|I^\gm_X f\right\|_{L^{q(\cdot)}_\rho} \le C \left\|f\right\|_{L^{p(\cdot)}_\rho}
\end{equation}
with the limiting exponent $q(\cdot)$ defined by $\frac{1}{q(x)}=\frac{1}{p(x)}-\gm$,
 holds if
\begin{equation}\label{cond}
\varrho^{-q_0} \in
 \mathfrak{A}_{\left(\frac{q(\cdot)}{q_0}\right)^\prime}(X)
\end{equation}
under any choice of $q_0>\frac{p_-}{1-\gm p_-}$.
\end{theorem}
\begin{proof} By Theorem \ref{thKokil}, inequality (\ref{admits}) holds under condition (\ref{muck-wheed}).
As is known,  inequality (\ref{new1}) with $f=I^\al g$ holds for every weight $w$ satisfying the
 $1<p_0<\infty $ and $\frac{1}{q_0}=\frac{1}{p_0}-\gm$. Condition (\ref{muck-wheed}) is
satisfied if $v^{q_0}\in A_1$. Consequently, inequality (\ref{new1}) with $f=I^\al g$ holds for every $w\in A_1$.
Then (\ref{Sobolev1})  follows  from  Theorem \ref{th4.1.}.
\end{proof}

From Theorem \ref{thpoten} we derive the following corollary for the Riesz potential operators
\begin{equation}\label{Riesz}
I^\al f(x)= \intl_{\rn}  \frac{f(y)\, dy}{|x-y|^{n-\al}}.
\end{equation}
\begin{corollary}\label{Riesz} Let $p\in \mathcal{P}$, let
$0<\al<n$ and $p_+<\frac{n}{\al}$. The weighted Sobolev theorem
\begin{equation}\label{Sobolev}
\left\|I^\al f\right\|_{L^{q(\cdot)}_\rho} \le C
\left\|f\right\|_{L^{p(\cdot)}_\rho}
\end{equation}
with the limiting exponent $q(\cdot)$ defined by
$\frac{1}{q(x)}=\frac{1}{p(x)}-\frac{\al}{n}$,
 holds if
\begin{equation}\label{cond}
\varrho^{-q_0} \in
 \mathfrak{A}_{\left(\frac{q(\cdot)}{q_0}\right)^\prime}(\mathbb{R}^n)
\end{equation}
under any choice of $q_0>\frac{np_-}{n-\al p_-}$.
\end{corollary}

\begin{remark}\label{remho}
Since Theorems \ref{th3.1.} and \ref{th3.2.} provide sufficient
conditions for the  weight $\varrho$ to satisfy assumption  (\ref{cond}), we could
write down the corresponding statements  on the validity of
(\ref{Sobolev}) in terms of the weights used in Theorems \ref{th3.1.} and \ref{th3.2.}. In the sequel we give results of such a kind for
other operators. For potential operators in
the case
$\Om=\rn$ we refer to \cite{584a} and \cite{539h}, where
for power weights  of the class $V_{p(\cdot)}(\rn,\Pi)$ and
for radial oscillating weights of the class
$V^{osc}_{p(\cdot)}(\rn,\Pi)$, respectively, there were obtained
estimates (\ref{Sobolev}) under assumptions more general than
should be imposed by the usage of Theorem \ref{th3.2.}.
\end{remark}

\subsection{Fourier multipliers}\label{subs5.1.}

A measurable function $\mathbb{R}^n\to \mathbb{R}^1$ is said to be
a Fourier multiplier in the space
$L_\varrho^{p(\cdot)}(\mathbb{R}^n)$, if the operator $T_m$,
defined on the Schwartz space $S(\mathbb{R}^n)$ by
$$\widehat{T_m f} = m \widehat{f},$$
admits an extension to the bounded operator  in
$L_\varrho^{p(\cdot)}(\mathbb{R}^n)$.

We give below a generalization of the classical Mikhlin theorem
(\cite{399a}, see also \cite{400})
 on Fourier multipliers to the case of Lebesgue spaces with variable
 exponent.
\begin{theorem} \label{th5.1.}
Let a function  $m(x)$ be continuous everywhere in $\mathbb{R}^n$,
except for probably the origin, have  the mixed distributional
derivative  $\frac{\partial^n m}{\partial x_1x_2\cdots x_n}$ and
  the derivatives   $D^\al m
=\frac{\partial^{|\al |} m}{\partial x_1^{\al _1}x_2^{\al
_2}\cdots x_n^{\al _n}}, \al =(\al _1,...,\al _n)$ of orders
 $|\al |=\al _1+\cdots_+\al _n\le n-1$ continuous beyond the origin and
$$|x|^{|\al |} |D^\al  m(x)|\le C,  \quad |\al |\le n-1,$$
where the constant  $C>0$ does not depend on $x$.  Then under
conditions (\ref{new1bcxc54esa}) and (\ref{vgschi}) with
$\Om=\mathbb{R}^n$, $m$ is a Fourier multiplier in
$L^{p(\cdot)}_\varrho(\mathbb{R}^n)$.
\end{theorem}

\begin{proof} Theorem \ref{th5.1.}  follows from  Theorem \ref{th4.1.}  under the choice
$\Om=X=\mathbb{R}^n$ and $\mathcal{F}=\{T_mg,g\}$ with $g\in S(\mathbb{R}^n)$, if we take into
account that in the case of constant $p_0 >1$ and weight $\varrho \in A_{p_0} (\supset A_1)$, a
function $m$, satisfying the assumptions of Theorem \ref{th5.1.}, is a Fourier multiplier in
$L_\varrho^{p_0}(\mathbb{R}^n)$. The latter was proved in \cite{349b}, see also \cite{316zz}.
\end{proof}

\begin{corollary}\label{cor}
Let $m$ satisfy the assumptions of Theorem \ref{th5.1.} and let
the exponent $p$ and the weight $\varrho$ satisfy the assumptions\\
i)  $p\in \mathcal{P}(\mathbb{R}^n)\cap WL(\mathbb{R}^n)$ and
$p(x)=p_\infty=const$ for $|x|\ge R$ with some  $R>0$,\\
ii) $\varrho\in V^{osc}_{p(\cdot)}(\mathbb{R}^n,\Pi), \Pi
=\{x_1,...x_N\}\subset
\mathbb{R}^n$.\\
Then $m$ is a Fourier multiplier in
$L^{p(\cdot)}_\varrho(\mathbb{R}^n)$.
 In
particular, assumption ii) holds for weights $\varrho$ of form
\begin{equation}\label{doch}
\varrho(x)= (1+|x|)^{\bt_\infty}\prod\limits_{k=1}^N
|x-x_k|^{\bt_k}, \ \ \  x_k\in \mathbb{R}^n,
\end{equation}
 where
\begin{equation}\label{5.1}
-\frac{n}{p(x_k)}<\bt_k< \frac{n}{p^\prime(x_k)}, \quad
k=1,2,...,N,
\end{equation}
\begin{equation}\label{5.2}
-\frac{n}{p_\infty}<\bt_\infty + \sum\limits_{k=1}^N\bt_k<
\frac{n}{p^\prime_\infty}.
\end{equation}
\end{corollary}
\begin{proof}
It suffices to observe that conditions on the weight $\varrho$
imposed in  Theorem \ref{th5.1.}, are fulfilled for $\varrho\in
V^{osc}_{p(\cdot)}(\mathbb{R}^n,\Pi)$ which follows from Remark
\ref{rem4.3.} and Theorem \ref{th3.2.}. In the  case of power
weights, conditions defining the class
$V^{osc}_{p(\cdot)}(\mathbb{R}^n,\Pi)$ turn into
(\ref{5.1})-(\ref{5.2}).
\end{proof}

The statement of Theorem \ref{th5.1.}  also holds in a more general form of Mikhlin/H\"ormander
theorem.

\vspace{4mm}\begin{theorem}\label{th5.2.} \textit{Let a function
$m:\mathbb{R}^n\to \mathbb{R}^1$ have distributional derivatives
up to order $\ell
>\frac{n}{p_-}$ satisfying the condition}
$$\sup\limits_{R>0}\left(R^{s|\al |-n}\intl_{{R<|x|<2R}}|D^\al m(x)|^sdx\right)^\frac{1}{s}<\infty$$
\textit{for some  $s, 1<s\le 2$ and all $\al $ with $|\al |\le
\ell.$ If conditions (\ref{new1bcxc54esa}), (\ref{vgschi}) with
$\Om=X=\mathbb{R}^n$ on $p$ and $\varrho$ are satisfied, then $m$
is a Fourier multiplier in $L^{p(\cdot)}_\varrho(\mathbb{R}^n)$.}
\end{theorem}

\begin{proof} Theorem \ref{th5.2.} is similarly derived from from Theorems
\ref{th4.1.} , if we take into account that  in the case of constant $p_0$ the statement of the
theorem  for Muckenhoupt weights was proved in \cite{Kurtznew}.
\end{proof}
\begin{corollary}\label{cor1} Let a function
$m:\mathbb{R}^n\to \mathbb{R}^1$ satisfy the assumptions of
Theorem \ref{th5.2.} and let $p$ and $\rho$ satisfy conditions
$i)$ and $ii)$ of Corollary \ref{cor}. Then $m$ is a Fourier
multiplier in $L^{p(\cdot)}_\varrho(\mathbb{R}^n)$.
\end{corollary}
\begin{proof} Follows from Theorem \ref{th5.2.} since conditions on the weight $\varrho$
imposed in  Theorem \ref{th5.1.}, are fulfilled for $\varrho\in
V^{osc}_{p(\cdot)}(\mathbb{R}^n,\Pi)$ by Theorem \ref{th3.2.} and
 Remark \ref{rem4.3.}.
\end{proof}

\vspace{3mm} In the next theorem  by  $\Delta_j$ we denote the
interval of the form $\Delta_j=[2^j,2^{j+1}]$ or
$\Delta_j=[-2^{j+1}, - 2^j], \ j\in \mathbb{Z}$.

\begin{theorem}\label{th5.3.} \textit{Let a function
$m:\mathbb{R}^1\to \mathbb{R}^1$ be representable in each interval
$\Delta_j$ as}
$$m(\lb)=\intl_{-\infty}^\lb d\mu_{\Delta_j}, \ \quad \lb\in \Delta_j,$$
\textit{where $\mu_{\Delta_j}$ are finite measures such that $\sup\limits_{j} \mathrm{var} \
\mu_{\Delta_j}<\infty$. If conditions (\ref{new1bcxc54esa}), (\ref{vgschi}) with
$\Om=X=\mathbb{R}^n$ on $p$ and $\varrho$ are satisfied, then $m$ is a Fourier multiplier in
$L^{p(\cdot)}_\varrho(\mathbb{R}^1)$.}
\end{theorem}

\begin{proof} To derive Theorem \ref{th5.3.} from Theorem \ref{th4.2.},  it suffices to refer to the boundedness
of the maximal operator in the   space $L^{p(\cdot)}_\varrho(\mathbb{R}^1)$ by Theorem \ref{th3.2.}
and the fact that in the case of constant $p$ the  theorem was proved in  \cite{368a} (for
$\varrho\equiv 1$) and \cite{316zz}, \cite{316zza} (for  $\varrho\in A_p$).
\end{proof}

\begin{corollary}\label{cor2}
Let $m$ satisfy the assumptions of Theorem \ref{th5.3.} and the
exponent $p$ and weight $\varrho$ fulfill conditions \textit{i)}
and \textit{ii)} of Corollary \ref{cor} with $n=1$. Then $m$ is a
Fourier multiplier in $L^{p(\cdot)}_\varrho(\mathbb{R}^1)$.
\end{corollary}

The "off-diagonal"$L^{p(\cdot)}_\varrho\to
L^{q(\cdot)}_\varrho$-version of Theorem \ref{th5.3.} in the case
 $q(x)>p(x)$ is covered by the following theorem.

\begin{theorem}\label{tg13.} Let
$p\in\mathcal{P}(\mathbb{R}^1)\cap WL(\mathbb{R}^1)$ and $p(x)\equiv p_\infty=const$ for large
$|x|>R,$ and let a function $m:\mathbb{R}^1\to \mathbb{R}^1$ be representable in each interval
$\Delta_j$ as
$$m(\lb)=\intl_{-\infty}^\lb\frac{d\mu_{\Delta_j}(t)}{(\lb-t)^\al}, \
\quad \lb\in \Delta_j,$$ where $ 0<\al<\frac{1}{p_+}$ and
$\mu_{\Delta_j}$ are the same as in Theorem \ref{th5.3.}. Then
$T_m$ is a bounded operator from
$L^{p(\cdot)}_\varrho(\mathbb{R}^1)$ to
$L^{q(\cdot)}_\varrho(\mathbb{R}^1)$, where
$$\frac{1}{q(x)}=\frac{1}{p(x)}-\al$$
and $\varrho$ is a weight of form (\ref{doch}) whose exponents
satisfy the conditions
\begin{equation}\label{5tyrok}
\al -\frac{1}{p(x_k)}<\bt_k< \frac{1}{p^\prime(x_k)}, \quad
k=1,2,...,N,   \quad \textrm{and} \quad \al
-\frac{1}{p_\infty}<\bt_\infty + \sum\limits_{k=1}^N\bt_k<
\frac{1}{p^\prime_\infty}.
\end{equation}
\end{theorem}

\begin{proof} In \cite{316zzb} there was proved that the operator
$T_m$ is bounded from $L^{p_0}_v(\rone)$ into $L^{q_0}_v(\rone)$
for every $p_0\in(1,\infty)$, $0<\al<\frac{1}{p_0}$,
$\frac{1}{q_0}=\frac{1}{p_0}-\al$, and an arbitrary weight $v$
satisfying the condition
\begin{equation}\label{6tyrok}
\sup\limits_{I}\left(\frac{1}{|I|}\intl_{I}v^{q_0}(x)dx\right)^\frac{1}{q_0}
\left(\frac{1}{|I|}\intl_{I}v^{-p^\prime_0}(x)dx\right)^\frac{1}{p_0},
\end{equation}
where the supremum is taken with respect to all one-dimensional
intervals. Condition (\ref{6tyrok}) is satisfied if $v^{q_0}\in
A_1$.  Then inequality (\ref{new1}) with $f=T_m g$ holds for every
$w\in A_1$. Then the statement of the theorem follows immediately
from Part \textit{II} of Theorem \ref{th4.2.}, conditions
(\ref{jui7})-(\ref{jun70}) turning into (\ref{5tyrok}) since
$\underline{\mathfrak{dim}}(\Om)=\underline{\mathfrak{dim}}_\infty(\Om)=1$,
$m(w_k)=M(w_k)=\bt_k, k=1,\dots, N$, and
$m(w_0)=M(w_0)=\bt_\infty.$
\end{proof}

\vspace{3mm} All the statements in the following subsections are
also  similar direct consequences of the general statement of
Theorem \ref{th4.2.} and Theorems \ref{th3.1.} and \ref{th3.2.} on
the maximal operator in the spaces $L^{p(\cdot)}_\varrho$, so that
in the sequel  for the proofs we only make references to  where
these statements were proved in the case of constant $p$ and
Muckenhoupt weights.

\subsection{Multipliers of trigonometric Fourier series}\label{subs5.2.}

With the help of Theorem \ref{th4.2.} and known results for
constant exponents, we are now able to give a generalization of
theorems on Marcinkiewicz multipliers and Littlewood-Paley
decompositions for trigonometric Fourier series to the case of
weighted  spaces with variable exponent.

Let $\mathbb{T}=[\pi,\pi]$ and let $ f$ be a $2\pi$-periodic
function and
\begin{equation}\label{5.3}
f(x) \sim \frac{a_0}{2} +\sum\limits_{k=0}^\infty (a_k \cos kx + b_k \sin kx).
\end{equation}

\begin{theorem} \label{th5.4.} Let a sequence
 $\lb_k$ satisfy the conditions
\begin{equation}\label{5vhtq}
|\lb_k|\le A \  \quad \quad \textrm{and}  \quad \quad
\sum_{k=2^{j-1}}^{{2^j}-1} |\lb_k-\lb_{k+1}|\le A,
\end{equation}
where $A>0$ does not depend on  $k$ and $j$.  Suppose that
\begin{equation}\label{nef41qs}
p \in \mathcal{P}(\mathbb{T}) \quad \textrm{and}\ \quad
\varrho^{-p_0} \in  \mathfrak{A}_{(\widetilde{p})^\prime}(\mathbb{T}), \quad \textrm{where}
 \quad  \widetilde{p}(\cdot)=\frac{p(\cdot)}{p_0}
\end{equation}
with some $p_0\in\left(1,p_-(\mathbb{T})\right)$. Then there
exists a function $F(x)\in L^{p(\cdot)}_\varrho(\mathbb{T})$ such
that the series $\frac{\lb_0 a_0}{2} +\sum\limits_{k=0}^\infty
\lb_k(a_k \cos kx + b_k sin kx)$ is Fourier series for  $F$ and
$$\|F\|_{L^{p(\cdot)}_\varrho}\le cA \|f\|_{L^{p(\cdot)}_\varrho}$$
where $c>0$ does not depend on  $f\in
L^{p(\cdot)}_\varrho(\mathbb{T})$.
\end{theorem}
\begin{corollary}\label{cor3} The statement   of Theorem
\ref{th5.4.} remains valid if condition (\ref{nef41qs}) is replaced by the assumption, sufficient for
(\ref{nef41qs}), that $p\in \mathcal{P}(\mathbb{T})\cap WL(\mathbb{T})$ and  the weight
 $\varrho$ has form
\begin{equation}\label{2.4asew}
\varrho(x)=\prod_{k=1}^N w_k(|x-x_k|), \quad x_k\in \mathbb{T}
\end{equation}
where
\begin{equation}\label{2.6cce}
w_k\in \widetilde{U}([0,2\pi])  \quad \textrm{and}\ \
-\frac{1}{p(x_k)} < m(w_k)\le M(w_k) < \frac{1}{p^\prime(x_k)} .
\end{equation}
\end{corollary}

\begin{theorem} \label{th5.5.}
Let
\begin{equation}\label{2bfe}
A_k(x)=a_k \cos kx + b_k \sin kx, \quad k=0,1,2,... , \quad
A_{2^{-1}}=0.
\end{equation}
 Under conditions (\ref{nef41qs}) there exist
constants $c_1>0$ and $c_2>0$ such that
\begin{equation}\label{5.4ax}
c_1\|f\|_{L^{p(\cdot)}_\varrho}\le
\left\|\left(\sum\limits_{j=0}^\infty\left|\sum\limits_{k=2^{j-1}}^{2^j-1}
A_k(x)\right|^2\right)^\frac{1}{2}\right\|_{L^{p(\cdot)}_\varrho}\le
c_2\|f\|_{L^{p(\cdot)}_\varrho}
\end{equation}
for all $f\in L^{p(\cdot)}_\varrho(\mathbb{T})$.
\end{theorem}

\vspace{2mm} In the case of constant  $p$ and $\varrho\in A_p$
this theorem was proved in \cite{349b}.

\begin{corollary}\label{cor4} Inequalities (\ref{5.4ax}) hold for
$p\in \mathcal{P}(\mathbb{T})\cap WL(\mathbb{T})$ and weights
$\varrho$ of form (\ref{2.4asew})-(\ref{2.6cce}).
\end{corollary}

\vspace{4mm}

\subsection{Majorants of partial sums of Fourier series} \label{subs5.3.}
Let
$$S_\ast(f)=S_\ast(f,x)=\sup\limits_{k\ge 0} |S_k(f,x)|,$$
where $S_k(f,x)=\sum\limits_{j=0}^kA_j(x)$ is a partial sum of
Fourier series (\ref{5.3}).

\begin{theorem} \label{th5.6.} Under conditions (\ref{nef41qs})
\begin{equation}\label{5.4}
\|S_\ast(f)\|_{L^{p(\cdot)}_\varrho}\le c
\|f\|_{L^{p(\cdot)}_\varrho},
\end{equation}
for all  $f\in L^{p(\cdot)}_\varrho(\mathbb{T})$, where the
constant  $c>0$ does not depend on $f$.
\end{theorem}

\vspace{4mm}In the case of constant  $p$ and $\varrho\in A_p$,
Theorem \ref{th5.6.} was proved in  \cite{236a}.

\begin{corollary}\label{cor5} Inequality (\ref{5.4}) is valid for
$p\in \mathcal{P}(\mathbb{T})\cap WL(\mathbb{T})$ and weights
$\varrho$ of form (\ref{2.4asew})-(\ref{2.6cce}).
\end{corollary}

 \subsection{Zygmund and Cesaro summability for trigonometric series in
 $L_\varrho^{p(\cdot)}(\mathbb{T})$} \label{subs5.u4.}

Under notation (\ref{5.3}) and (\ref{2bfe}) we introduce the
Zygmund and Cesaro means of summability
$$Z_n^{(2)}(f,x)= \sum\limits_{k=0}^n \left[1-
\left(\frac{k}{n+1}\right)^2\right]A_k(x)$$ and
$$\sg_n(f,x)=\frac{1}{n+1}\sum\limits_{k=0}^n S_k(f,x),$$
respectively. By
$$\Om_{p,\varrho}(f,\dl)=\sup\limits_{0<h<\dl}\|(I-\tau_h)f\|_{L_\varrho^{p(\cdot)}}$$
we denote the continuity modulus of a function $f$ in
$L_\varrho^{p(\cdot)}(\mathbb{T})$ with respect to the generalized
shift (Steklov mean)
$$\tau_h f(x)=\frac{1}{2h}\intl_{x-h}^{x+h} f(t) dt.$$
\begin{theorem}\label{thnew}
Under conditions (\ref{nef41qs}) there hold the estimates
\begin{equation}\label{43kj}
\|f(\cdot)-Z_n^{(2)}(f,\cdot)\|_{L_\varrho^{p(\cdot)}} \le
C\Om_{p,\varrho}\left(f,\frac{1}{n}\right)
\end{equation}
and
\begin{equation}\label{43kjbc}
\|f(\cdot)-\sg_n(f,\cdot)\|_{L_\varrho^{p(\cdot)}} \le
Cn\Om_{p,\varrho}\left(f,\frac{1}{n}\right).
\end{equation}
\end{theorem}
\begin{proof}
We make use  of the estimate
\begin{equation}\label{4vr5}
\|f(\cdot)-S_n(f,\cdot)\|_{L_\varrho^{p(\cdot)}} \le
C\Om_{p,\varrho}\left(f,\frac{1}{n}\right)
\end{equation}
proved in \cite{IKS} under assumptions (\ref{nef41qs}). For the
difference $S_n(f,x)-Z_n^{(2)}(f,x)$ we have
\begin{equation}\label{4nevr5}
\|S_n(f,\cdot)-Z_n^{(2)}(f,\cdot)\|_{L_\varrho^{p(\cdot)}}=
\left\|\sum\limits_{k=1}^n
\left(\frac{k}{n+1}\right)^2A_k(\cdot)\right\|_{L_\varrho^{p(\cdot)}}.
\end{equation}
 Keeping in mind that
\begin{equation}\label{4bv9a}
f(x)-\tau_h f(x)\sim \sum\limits_{k=1}^\infty\left(1-\frac{\sin
kh}{kh}\right)A_k(x),
\end{equation}
 we transform (\ref{4nevr5}) to
$$\|S_n(f,\cdot)-Z_n^{(2)}(f,\cdot)\|_{L_\varrho^{p(\cdot)}}
= \left\|\sum\limits_{k=1}^n \lb_{k,n} \left(1-\frac{\sin
\frac{k}{n}}{\frac{k}{n}}\right)A_k(\cdot)\right\|_{L_\varrho^{p(\cdot)}}$$
where
$$\lb_{k,n}=\left\{\begin{array}{cc}\frac{\left(\frac{k}{n+1}\right)^2}
{1-\frac{\sin \frac{k}{n}}{\frac{k}{n}}}, & k\le n
\\
 0,  & k>n\end{array}\right.$$
It is easy to check that the multiplier $\lb_{k,n}$ satisfies
assumptions (\ref{5vhtq}) of Theorem \ref{th5.4.} with the
constant $A$ in (\ref{5vhtq}) not depending on $n$. Therefore, by
Theorem \ref{th5.4.} we get
$$\|S_n(f,\cdot)-Z_n^{(2)}(f,\cdot)\|_{L_\varrho^{p(\cdot)}}
\le C\left\|\sum\limits_{k=1}^\infty  \left(1-\frac{\sin
\frac{k}{n}}{\frac{k}{n}}\right)A_k(\cdot)\right\|_{L_\varrho^{p(\cdot)}} = C\left\|f-\tau_h
f\right\|_{L_\varrho^{p(\cdot)}}$$ by (\ref{4bv9a}). Then in view of (\ref{4vr5}) estimate
(\ref{43kj}) follows.

Estimate (\ref{43kjbc}) is similarly obtained,  with the
multiplier $\lb_{k,n}$ of the form
$$\left\{\begin{array}{cc}\frac{\frac{k}{n+1}} {n\left(1-\frac{\sin
\frac{k}{n}}{\frac{k}{n}}\right)}, & k\le n
\\
 0,  & k>n\end{array}\right..$$
\end{proof}
\begin{corollary}\label{cor5vc} Estimates  (\ref{43kj}),(\ref{43kjbc}) are valid for
$p\in \mathcal{P}(\mathbb{T})\cap WL(\mathbb{T})$ and weights
$\varrho$ of form (\ref{2.4asew})-(\ref{2.6cce}).
\end{corollary}
\begin{remark}\label{643}
When  $p>1$ is constant, estimates (\ref{43kj}),(\ref{43kjbc}) in
the non-weighted  case were obtained in \cite{Knew}.
\end{remark}

 \subsection{Cauchy singular integral} \label{subs5.4.}

We consider the singular integral operator
$$S_\Gm f(t)=\frac{1}{\pi i} \intl_\Gm\frac{f(\tau)\, d\nu(\tau)}{\tau-t},$$
where $\Gm$ is a simple finite Carleson curve and  $\nu$ is an arc
length.

\begin{theorem} \label{th5.7.} Let
\begin{equation}\label{nef41}
p \in \mathcal{P}(\Gm) \quad \textrm{and}\ \quad
\varrho^{-p_0}\in
 \mathfrak{A}_{(\widetilde{p})^\prime}(\Gm)
\end{equation}
for some $p_0\in (1,p_-)$, where
 $ \widetilde{p}(\cdot)=\frac{p(\cdot)}{p_0}$.
 Then the
operator  $S_\Gm$ is bounded in the space $L^{p(\cdot)}_\varrho
(\Gm)$ .
\end{theorem}

 For  the case of constant
   $p$ and  $\varrho^p\in A_p(\Gm)$, Theorem \ref{th5.7.} by different methods was proved in  \cite{310a} and
   \cite{63}. (As is known, $\varrho^{-p_0}\in
 \mathfrak{A}_{(\widetilde{p})^\prime}(\Gm))\Longleftrightarrow \varrho^p\in
 A_{\frac{p}{p_0}}(\Gm)$ for an
 arbitrary Carleson curve in the case of constant $p$,
  see \cite{310a} and \cite{63},  so that the
   conditions $\varrho^{-p_0}\in
 \mathfrak{A}_{(\widetilde{p})^\prime}(\Gm))$
and $\varrho^p\in A_p(\Gm)$ are equivalent  in the sense that the former always yields the latter
for every $p_0>1$ and the latter
  yields the former for some $p_0>1$).

\begin{corollary}\label{cor6}  The
operator  $S_\Gm$ is bounded in the space $L^{p(\cdot)}_\varrho
(\Gm)$, if  $p\in \mathcal{P}(\Gm)\cap WL(\Gm)$ and the  weight
$\varrho$ has the  form

\begin{equation}\label{2.4axcsew}
\varrho(t)=\prod_{k=1}^N w_k(|t-t_k|), \quad t_k\in \Gm,
\end{equation}where
\begin{equation}\label{2.6ccece}
w_k\in \widetilde{U}([0,\nu(\Gm)])  \quad \textrm{and}\ \
-\frac{1}{p(t_k)} < m(w_k)\le M(w_k) < \frac{1}{p^\prime(t_k)} .
\end{equation}

\end{corollary}

In the case of power weights, the statement of Corollary
\ref{cor6}  was proved in
 \cite{317b}, where the case of an infinite Carleson curve was also dealt with.

\subsection{Multidimensional singular operators}\label{subs5.5.}

\vspace{3mm}

We consider a multidimensional  singular operator
\begin{equation}\label{5.40}
Tf(x)=\lim\limits_{\ve\to 0}\intl\limits_{y\in\Om : |x-y|>\ve}
K(x,y) f(y)\,dy, \quad x\in \Om\subseteq \mathbb{R}^n,
\end{equation}
where we assume that the singular kernel $K(x,y)$ satisfies the
assumptions:
\begin{equation}\label{5.4a}
|K(x,y)|\le C |x-y|^{-n},
\end{equation}
\begin{equation}\label{5.4b}
|K(x^\prime,y)-K(x,y)|\le C \frac{|x^\prime
-x|^\al}{|x-y|^{n+\al}}, \quad \quad |x^\prime
-x|<\frac{1}{2}|x-y|,
\end{equation}
\begin{equation}\label{5.4c}
|K(x,y^\prime)-K(x,y)|\le C \frac{|y^\prime
-y|^\al}{|x-y|^{n+\al}}, \quad \quad |y^\prime
-y|<\frac{1}{2}|x-y|,
\end{equation}
where $\al$ is an arbitrary positive exponent,
\begin{equation}\label{5.4ca}
\textrm{there exists} \ \ \ \lim\limits_{\ve\to 0}\intl_{y\in\Om:
 |x-y|>\ve} K(x,y)\, dy,
\end{equation}

\begin{equation}\label{5.4d}
\textrm{operator (\ref{5.40}) is bounded in $L^2(\Om)$. }
\end{equation}

 \begin{theorem} \label{th5.8.}  Let the
kernel  $K(x,y)$ fulfill conditions (\ref{5.4a})-(\ref{5.4d}).
Then under the conditions
\begin{equation}\label{nef41f4}
p \in \mathcal{P}(\Om) \quad \textrm{and}\ \quad
\varrho^{-p_0}\in
 \mathfrak{A}_{(\widetilde{p})^\prime}(\Om) \quad \textrm{with}
 \quad  \widetilde{p}(\cdot)=\frac{p(\cdot)}{p_0}
\end{equation}
  the operator $T$ is bounded in the
space $L^{p(\cdot)}_\varrho (\Om)$.
\end{theorem}

In the case of constant  $p$ and $\varrho\in A_p(\mathbb{R}^n)$,
Theorem \ref{th5.8.} was proved in \cite{100z}.

\begin{corollary}\label{cor7} Let $p\in
\mathcal{P}(\Om)\cap WL(\Om)$ and let $p(x)\equiv p_\infty=const $
outside some ball  $|x|< R$ in case $\Om$ is unbounded.  The
operator  $T$ with the kernel satisfying conditions
(\ref{5.4a})-(\ref{5.4d}) is bounded in the space
$L^{p(\cdot)}_\varrho (\Om)$ with a weight $\varrho$ of the form
\begin{equation}\label{2.4agggsw}
\varrho(x)=\prod_{k=1}^N w_k(|x-x_k|), \quad x_k\in \Om,
\end{equation}where
$ w_k\in \widetilde{U}(\mathbb{R}_+^1)$ and    $$ -\frac{1}{p(x_k)} < m(w_k)\le M(w_k) < \frac{1}{p^\prime(x_k)}
\quad \textrm{and} \quad  -\frac{n}{p_\infty}<\sum\limits_{k=1}^N m_\infty(w_k)\le \sum\limits_{k=1}^N
M_\infty(w_k) <\frac{n}{p^\prime_\infty}. $$

\end{corollary}

In the case of variable
 $p(\cdot)$, the statement of Corollary  \ref{cor7} was proved in
 \cite{107a} in the non-weighted case, and in
 \cite{317e} in weighted case (\ref{2.4agggsw}) for bounded sets $\Om$.

\subsection{Commutators}\label{subs5.6.}

Let us consider the commutators
$$[b,T]f(x)=b(x)Tf(x)-T(bf)(x), \quad x\in\mathbb{R}^n$$
generated by the operator (\ref{5.40}) with $\Om=\mathbb{R}^n$ and
a function $b\in BMO(\mathbb{R}^n)$.

 \begin{theorem} \label{th5.9.}
  Let the kernel  $K(x,y)$
fulfill assumptions (\ref{5.4a})-(\ref{5.4d}) and let $b\in
BMO(\mathbb{R}^n)$. Then under the  conditions
\begin{equation}\label{nefgoa4}
p \in \mathcal{P}(\mathbb{R}^n) \quad \textrm{and}\ \quad
\varrho^{-p_0}\in
 \mathfrak{A}_{(\widetilde{p})^\prime}(\mathbb{R}^n) \quad \textrm{with}
 \quad  \widetilde{p}(\cdot)=\frac{p(\cdot)}{p_0}
\end{equation}
 the commutator
$[b,T]$ is bounded in the space $L^{p(\cdot)}_\varrho
(\mathbb{R}^n)$.
\end{theorem}

In the case of constant  $p$ and  $\varrho\in A_p(\mathbb{R}^n),
1<p<\infty$, Theorem \ref{th5.9.} was proved in  \cite{479zz}. In
the case of variable  $p(\cdot)$, the non-weighted case of Theorem
\ref{th5.9.} was proved in  \cite{299c} under the assumption that
$1\in \mathfrak{A}_{p(\cdot)}(\mathbb{R}^n)$.

\begin{corollary}\label{cor8}
 Let the kernel $K(x,y)$ fulfill
conditions (\ref{5.4a})-(\ref{5.4d}) and let $b\in
BMO(\mathbb{R}^n)$. Then the commutator $[b,T]$ is bounded in the
space $L^{p(\cdot)}_\varrho (\mathbb{R}^n)$ if\\
i) \  $p\in \mathcal{P}(\mathbb{R}^n)\cap WL(\mathbb{R}^n)$ and
$p(x)\equiv
p_\infty=const $ outside some ball  $|x|< R$, \\
2) the weight $\varrho$ has  the form
\begin{equation}\label{2.4agggsw}
\varrho(x)=w_0(1+|x|)\prod_{k=1}^N w_k(|x-x_k|), \quad x_k\in
\mathbb{R}^n,
\end{equation} with the factors $w_k, \ k=0,1,...,N,$ satisfying conditions (\ref{2.6bvc})-(\ref{f27dcobvc}).
\end{corollary}

\subsection{Pseudo-differential operators} \label{subs5.8.}

We consider a pseudo-differential operator $\sg(x,D)$  defined by
$$\sg(x,D) f(x)=\intl_{\mathbb{R}^n}\sg(x,\xi)e^{2\pi i(x,\xi)}\hat{f}(\xi)\,d\xi.$$

 \begin{theorem} \label{th5.11.}  Let the symbol $\sg(x,\xi)$
 satisfy the condition
$$
\left|\partial^\al_\xi\partial_x^\bt \sg(x,\xi)\right|\le
c_{\al\bt}(1+|\xi|)^{-|\al|}
$$
for all the multiindices  $\al$ and $\bt$. Then under condition
(\ref{nefgoa4}) the operator $\sg(x,D)$ admits a continuous
extension to the space $L^{p(\cdot)}_\varrho (\mathbb{R}^n)$.
\end{theorem}

 \vspace{4mm}
In the case of constant  $p$ and $\varrho\in A_p$ Theorem
\ref{th5.11.} was proved in  \cite{404a}.

\begin{corollary}\label{cor9}
Let $p\in \mathcal{P}(\mathbb{R}^n)\cap WL(\mathbb{R}^n)$ and
$p(x)\equiv p_\infty=const $ outside some ball  $|x|< R$ and let
$\varrho \in
 V^{osc}_{p(\cdot)}(\mathbb{R}^n,\Pi)$.
\end{corollary}
For variable $p(\cdot)$ the statement of Corollary \ref{cor9} by a
different method was proved in the non-weighted case in
\cite{503a}.

\subsection{Feffermann-Stein function} \label{subs5.7.}

Let $f$  be a measurable locally integrable function  on
$\mathbb{R}^n$, $B$ an arbitrary ball in $\mathbb{R}^n$, \
$f_B=\frac{1}{|B|}\intl_{B}f(x)\,dx$ and
$$\mathcal{M}^\# f(x)=\sup\limits_{B\in X}\frac{1}{|B|}\intl_{B}|f(x)-f_B|\,dx $$
be  the Fefferman-Stein maximal function.

 \begin{theorem}  \label{th5.10.} Under condition
(\ref{nefgoa4}), the inequality
\begin{equation}\label{5.5}
\|\mathcal{M}f\|_{L^{p(\cdot)}_\varrho (\mathbb{R}^n)}\le C
\|\mathcal{M}^\# f\|_{L^{p(\cdot)}_\varrho (\mathbb{R}^n)}
\end{equation}
is valid, where $C>0$ does not depend on  $f$.
\end{theorem}

In the case of constant  $p$ and $\varrho\in A_p$ inequality
(\ref{5.5}) was proved in \cite{160c}.

\begin{corollary}\label{cor10} Inequality (\ref{5.5}) is valid under the conditions:
\\
i) $p\in \mathcal{P}(\mathbb{R}^n)\cap WL(\mathbb{R}^n)$ and
$p(x)\equiv p_\infty=const $ outside some ball   $|x|< R$, \\
ii)   $\varrho \in
 V^{osc}_{p(\cdot)}(\mathbb{R}^n,\Pi)$.
\end{corollary}

\subsection{Vector-valued operators}\label{subs5.9.}

Let $f=(f_1,\cdots,f_k, \cdots)$, where
$f_i:\mathbb{R}^n\to\mathbb{R}^1$  are locally integrable
functions.

 \begin{theorem} \label{th5.12.} Let $0<\theta<\infty$. Under
 conditions (\ref{nefgoa4}), the inequality
\begin{equation}\label{x}
\left\|\left(\sum\limits_{j=1}^\infty(\mathcal{M}
f_j)^\theta\right)^\frac{1}{\theta}\right\|_{L^{p(\cdot)}_\varrho(\mathbb{R}^n)}
\le C
\left\|\left(\sum\limits_{j=1}^\infty|f_j|^\theta\right)^\frac{1}{\theta}\right\|
_{L^{p(\cdot)}_\varrho(\mathbb{R}^n)}
\end{equation}
 is valid,
where $c>0$ does not depend on  $f$.
\end{theorem}

\vspace{5mm} In the case of  constant  $p$ and $\varrho\in A_p$
weighted inequalities for vector-valued functions were proved in
\cite{316zz}, \cite{316zza}, \cite{316zzb},  see also \cite{20a}.

\begin{corollary}\label{cor11} Inequality (\ref{x}) is valid under the conditions
\\
i) $p\in \mathcal{P}(\mathbb{R}^n)\cap WL(\mathbb{R}^n)$ and
$p(x)\equiv p_\infty=const $ outside some ball  $|x|< R$,\\
ii) $\varrho \in
 V^{osc}_{p(\cdot)}(\Om,\Pi)$.
\end{corollary}

\begin{remark}The corresponding statements for vector-valued
operators are also similarly derived from Theorem \ref{th4.2.} in
the case of singular integrals, commutators, Feffermann-Stein
maximal function, Fourier-multipliers, etc.
\end{remark}

\vspace{10mm}   This work was made under the project "Variable
Exponent Analysis" supported by INTAS grant
 Nr.06-1000017-8792. The first author was also  supported by Center CEMAT, Instituto Superior
T\'ecnico, Lisbon, Portugal, during his visit to Portugal,
November 29 -  December 2006.



\begin{thebibliography}{10}

\bibitem{9b}
E. Acerbi and G.Mingione
\newblock Regularity results for a class of functionals with
              non-standard growth.
\newblock {\em Arch. Ration. Mech. Anal.}, 156(2):121--140, 2001.


\bibitem{9d}
E. Acerbi and G.Mingione
\newblock Regularity results for staionary  electrorheological fluids.
\newblock {\em Arch. Ration. Mech. Anal.}, 164(3): 213--259, 2002.


\bibitem{20a}
K.~F. Andersen and R.~T. John.
\newblock Weighted inequalities for vector-valued maximal functions and
  singular integrals.
\newblock {\em Studia Math.}, 69(1):19--31, 1980/81.

\bibitem{46}
N.K. Bary and S.B. Stechkin.
\newblock Best approximations and differential properties of two conjugate
  functions (in {Russian}).
\newblock {\em Proceedings of Moscow Math. Soc.}, 5:483--522, 1956.

\bibitem{63}
A.~B\"ottcher and Yu. Karlovich.
\newblock {\em Carleson {Curves}, {Muckenhoupt} {Weights}, and {Toeplitz}
  {Operators}}.
\newblock Basel, {Boston}, {Berlin}: {Birkh\"auser} Verlag, 1997.
\newblock 397 pages.

\bibitem{72b}
A.-P. Calder{\'o}n.
\newblock Inequalities for the maximal function relative to a metric.
\newblock {\em Studia Math.}, 57(3):297--306, 1976.

\bibitem{97}
R.R. Coifman and G.~Weiss.
\newblock {\em Analyse harmonique non-commutative sur certaines espaces
  homegenes}, volume 242.
\newblock Lecture Notes Math., 1971.
\newblock 160 pages.

\bibitem{97a}
R.R. Coifman and G.~Weiss.
\newblock Extensions of {H}ardy spaces and their use in analysis, 1977.
\newblock {\em Bull. Amer. Math. Soc.}, 83(4):569--645, 1977.




\bibitem{100z}
A.~Cordoba and C.~Fefferman.
\newblock A weighted norm inequality for singular integrals.
\newblock {\em Studia Math.}, 57(1):97--101, 1976.

\bibitem{101zb}
D.~Cruz-Uribe, A.~Fiorenza, J.M. Martell, and C~Perez.
\newblock The boundedness of classical operators on variable {$L\sp p$} spaces.
\newblock {\em Ann. Acad. Sci. Fenn. Math.}, 31(1):239--264, 2006.

\bibitem{101ab}
D.~Cruz-Uribe, A.~Fiorenza, and C.J. Neugebauer.
\newblock The maximal function on variable ${L}^p$-spaces.
\newblock {\em Ann. Acad. Scient. Fennicae, Math.}, 28:223--238, 2003.

\bibitem{101ac}
D.~Cruz-Uribe, J.~M. Martell, and C.~P{\'e}rez.
\newblock Extrapolation from {$A\sb \infty$} weights and applications.
\newblock {\em J. Funct. Anal.}, 213(2):412--439, 2004.

\bibitem{106}
L.~Diening.
\newblock Maximal function on generalized {L}ebesgue spaces {$L\sp
  {p(\cdot)}$}.
\newblock {\em Math. Inequal. Appl.}, 7(2):245--253, 2004.

\bibitem{105a}
L.~Diening.
\newblock Riesz potential and {Sobolev} embeddings on generalized {Lebesgue}
  and {Sobolev} spaces ${L}^{p(\cdot)}$ and ${W}^{k, p(\cdot)}$.
\newblock {\em Mathem. Nachrichten}, 268:31--43, 2004.

\bibitem{106z}
L.~Diening.
\newblock Maximal function on {M}usielak-{O}rlicz spaces and generalized
  {L}ebesgue spaces.
\newblock {\em Bull. Sci. Math.}, 129(8):657--700, 2005.

\bibitem{106b}
L.~Diening, P.~H{\"a}st{\"o}, and A.~Nekvinda.
\newblock Open problems in variable exponent {Lebesgue} and {S}obolev spaces.
\newblock In {\em "Function Spaces, Differential Operators and Nonlinear
  Analysis", Proceedings of the Conference held in Milovy, Bohemian-Moravian
  Uplands, May 28 - June 2, 2004}. Math. Inst. Acad. Sci. Czech Republick,
  Praha.

\bibitem{107a}
L.~Diening and M.~Ru$\check{z}$i$\check{c}$ka.
\newblock Calderon-{Z}ygmund operators on generalized {Lebesgue} spaces
  ${L}^{p(x)}$ and problems related to fluid dynamics.
\newblock {\em J. Reine Angew. Math}, 563:197--220, 2003.

\bibitem{145a}
D.E. Edmunds, V.~Kokilashvili and A.~Meskhi.
\newblock {\em Bounded and Compact Integral Operators}, volume 543 of {\em
  Mathematics and its Applications}.
\newblock Kluwer Academic Publishers, Dordrecht, 2002.

\bibitem{160zzzz}
K.~Falconer.
\newblock {\em Techniques in fractal geometry}.
\newblock John Wiley \& Sons Ltd., Chichester, 1997.


\bibitem{160zb}
X. Fan and D. Zhao.
\newblock {A class of {D}e {G}iorgi type and {H}\"older continuity}.
\newblock {\em Nonlinear Anal.}, 36(3, Ser. A):295--318, 1999.




\bibitem{160c}
C.~Fefferman and E.~M. Stein.
\newblock {$H\sp{p}$} spaces of several variables.
\newblock {\em Acta Math.}, 129(3-4):137--193, 1972.

\bibitem{187}
I.~Genebashvili, A.~Gogatishvili, V.~Kokilashvili, and M.~Krbec.
\newblock {\em Weight Theory for Integral Transforms on Spaces of Homogeneous
Type.}
\newblock Pitman {Monographs} and {Surveys}, {Pure} and {Applied} mathematics:
  {Longman} {Scientific} and Technical, 1998.
\newblock 422 pages.

\bibitem{HMS}
E.~Harboure, R.A,~Macias and C.~Segovia.
\newblock Extrapolation {R}esults for {C}lasses of {W}eights.
\newblock {\em Amer. J. Math.}, 110(3): 383-397, 1988.

\bibitem{224ab}
P.~Harjulehto, P.~H{\"a}st{\"o}, and V.~Latvala.
\newblock Sobolev embeddings in metric measure spaces with variable dimension.
\newblock {\em Math. Z.}, 254(3):591--609, 2006.

\bibitem{224b}
P.~Harjulehto, P.~H{\"a}st{\"o}, and M.~Pere.
\newblock Variable {E}xponent {L}ebesgue {S}paces on {M}etric {S}paces: {T}he
  {H}ardy-{L}ittlewood {M}aximal {O}perator.
\newblock {\em Real Anal. Exchange}, 30(1):87--104, 2004.

\bibitem{225a}
J.~Heinonen.
\newblock {\em Lectures on Analysis on Metric Spaces}.
\newblock Universitext. Springer-Verlag, New York, 2001.

\bibitem{236a}
R.~A. Hunt and W.S. Young.
\newblock A weighted norm inequality for {F}ourier series.
\newblock {\em Bull. Amer. J. Math.},
\newblock 80:274--277, 1974.

\bibitem{IKS}
D.M.Israfilov, V.Kokilkashvili and S.Samko
\newblock Approximation in weighted Lebesgue spacesand Smirnov spaces
with variable exponents
\newblock
\newblock {\em Proc. A.Razmadze Math. Inst.}, 143: 25-35, 2007.

\bibitem{270a}
N.K. Karapetiants and N.G. Samko.
\newblock Weighted theorems on fractional integrals in the generalized
  {H}\"older spaces ${H}_0^\omega(\rho)$ via the indices $m_\omega$ and
  ${M}_\omega$.
\newblock {\em Fract. Calc. Appl. Anal.}, 7(4):437--458, 2004.

\bibitem{299c}
A.~Yu. Karlovich and A.K. Lerner.
\newblock Commutators of singular integrals on generalized {$L\sp p$} spaces
  with variable exponent.
\newblock {\em Publ. Mat.}, 49(1):111--125, 2005.

\bibitem{310a}
G.~Khuskivadze, V.~Kokilashvili, and V.~Paatashvili.
\newblock Boundary value problems for analytic and harmonic functions in
  domains with nonsmooth boundaries. {A}pplications to conformal mappings.
\newblock {\em Mem. Differential Equations Math. Phys.}, 14:195, 1998.


\bibitem{Knew}
V.~Kokilashvili.
\newblock On approximation of periodic functions (in {R}ussian).
\newblock {\em  Trudy A.Razmadze Mat. Inst.,  Akad.Nauk Gruzin. SSR},
34:51--81, 1968.



\bibitem{316b}
V.~Kokilashvili.
\newblock On a progress in the theory of integral operators in weighted
  {B}anach function spaces.
\newblock In {\em "Function Spaces, Differential Operators and Nonlinear
  Analysis", Proceedings of the Conference held in Milovy, Bohemian-Moravian
  Uplands, May 28 - June 2, 2004}. Math. Inst. Acad. Sci. Czech Republick,
  Praha.

\bibitem{316zz}
V.~Kokilashvili.
\newblock Maximal inequalities and multipliers in weighted {L}izorkin-{T}riebel
  spaces.
\newblock {\em Dokl. Akad. Nauk SSSR}, 239(1):42--45, 1978.

\bibitem{316zza}
V.~Kokilashvili.
\newblock Maximal functions in weighted spaces.
\newblock {\em Boundary properties of analytic functions, singular integral
  equations and some questions of harmonic analysis, Akad. Nauk Gruzin. SSR
  Trudy Tbiliss. Mat. Inst. Razmadze}, 65:110--121, 1980.

\bibitem{316zzb}
V.~Kokilashvili.
\newblock Weighted {L}izorkin-{T}riebel spaces. {S}ingular integrals,
  multipliers, imbedding theorems.
\newblock {\em Trudy Mat. Inst. Steklov. \ Studies in the theory of
  differentiable functions of several variables and its applications, IX},
  161:125--149, 1983.
\newblock English Transl. in Proc. Steklov Inst. Math. 3(1984), 135-162.

\bibitem{317b}
V.~Kokilashvili, V.~Paatashvili, and Samko S.
\newblock Boundedness in {L}ebesgue spaces with variable exponent of the
  {C}auchy singular operators on {C}arleson curves.
\newblock In Ya. Erusalimsky, I.~Gohberg, S.~Grudsky, V.~Rabinovich, and
  N.~Vasilevski, editors, {\em "Operator Theory: Advances and Applications",
  dedicated to 70th birthday of Prof. I.B.Simonenko}, pages 167--186.
  Birkh\"auser Verlag, Basel, 2006.


\bibitem{317c}
V.~Kokilashvili, N.~Samko, and S.~Samko.
\newblock The maximal operator in variable spaces
  ${L}^{p(\cdot)}({\Omega},\rho)$.
\newblock {\em Georgian Math. J.}, 13(1):109--125, 2006.

\bibitem{317e}
V.~Kokilashvili, N.~Samko, and S.~Samko.
\newblock Singular operators in variable spaces ${L}^{p(\cdot)}({\Omega},\rho)$
  with oscillating weights.
\newblock {\em Math. Nachr., 280(9-10): 1145-1156, 2007.}

\bibitem{JFSA}
V.~Kokilashvili, N.~Samko, and S.~Samko.
\newblock The {M}aximal {O}perator in {W}eighted {V}ariable {S}paces
  ${L}^{p(\cdot)}$.
\newblock {\em J. Function spaces and Appl.}
\newblock 5(3): 299-317, 2007.


\bibitem{321c}
V.~Kokilashvili and S.~Samko.
\newblock Singular {Integrals} in {Weighted} {Lebesgue} {Spaces} with
  {Variable} {Exponent}.
\newblock {\em Georgian Math. J.}, 10(1):145--156, 2003.

\bibitem{321a}
V.~Kokilashvili and S.~Samko.
\newblock Maximal and fractional operators in weighted ${L}^{p(x)}$ spaces.
\newblock {\em Revista Matematica Iberoamericana}, 20(2):495--517, 2004.

\bibitem{321i}
V.~Kokilashvili and S.~Samko.
\newblock {B}oundedness in {L}ebesgue spaces with variable exponent of maximal,
  singular and potential operators.
\newblock {\em Izvestija VUZov. Severo-Kavkazskii region. Estestvennie nauki,
  Special issue "Pseudodifferential equations and some problems of mathematical
  physics", dedicated to 70th birthday of Prof. I.B.Simonenko}, pages 152--158,
  2006.

\bibitem{newmetric}
V.~Kokilashvili and S.~Samko.
\newblock {The maximal operator in weighted variable spaces  on metric
measure spaces},
\newblock {\em Proc. A.Razmadze Math. Inst.,} 144:137-144, 2007.

\bibitem{321j}
V.~Kokilashvili and S.~Samko.
\newblock {Boundedness of maximal operators and potential operators on
{C}arleson curves  in {L}ebesgue spaces with variable exponent},
\newblock {\em Acta Mathematica Sinica}, 24(1), 2008.



\bibitem{332}
O.~Kov$\acute{\textrm{a}}$c$\check{\textrm{i}}$k and
  J.~R$\acute{\textrm{a}}$kosn$\check{\textrm{i}}$k.
\newblock On spaces ${L}^{p(x)}$ and ${W}^{k,p(x)}$.
\newblock {\em Czechoslovak {Math.} {J}.}, 41(116):592--618, 1991.

\bibitem{342}
S.G. Krein, Yu.I. Petunin, and E.M. Semenov.
\newblock {\em Interpolation of Linear Operators}.
\newblock Moscow: Nauka, 1978.
\newblock 499 pages.

\bibitem{342a}
S.G. Krein, Yu.I. Petunin, and E.M. Semenov.
\newblock {\em Interpolation of Linear Operators}, volume~54 of {\em
  Translations of Mathematical Monographs}.
\newblock American Mathematical Society, Providence, R.I., 1982.

\bibitem{349b}
D.~S. Kurtz.
\newblock Littlewood-{P}aley and multiplier theorems on weighted {$L\sp{p}$}\
  spaces.
\newblock {\em Trans. Amer. Math. Soc.}, 259(1):235--254, 1980.

\bibitem{Kurtznew}
D.~S. Kurtz and R.L. Wheeden.
\newblock Results on weighted norm inequalities for multipliers.
\newblock {\em Trans. Amer. Math. Soc.}, 255:343--562, 1979.



\bibitem{368a}
P.I. Lizorkin.
\newblock Multipliers of {F}ourier integrals in the spaces {$L\sb{p,\,\theta
  }$}.
\newblock {\em Trudy Mat. Inst. Steklov}, 89:231--248, 1967.
\newblock English Transl. in Proc. Steklov Inst. Math. 89 (1967), 269-290.

\bibitem{381a}
R.~Mac$\grave{i}$as and C.~Segovia.
\newblock A well behaved quasidistance for spaces of homogeneous type.
\newblock {\em Trab. Mat.Inst.Argentina Mat.}, 32:1--18, 1981.

\bibitem{382a}
L. Maligranda.
\newblock Indices and Interpolation.
\newblock {\em Dissertationes Math. (Rozprawy Mat.)}, 234:49, 1985.

\bibitem{382b}
L. Maligranda.
\newblock {\em Orlicz spaces and Interpolation}.
\newblock Departamento de Matem\'atica, Universidade Estadual de Campinas,
  1989.
\newblock Campinas SP Brazil.

\bibitem{399a}
S.G. Mikhlin.
\newblock On multipliers of {F}ourier integrals (in {R}ussian).
\newblock {\em Dokl. Akad. Nauk SSSR}, 109:701--703, 1956.

\bibitem{400}
S.G. Mikhlin.
\newblock {\em Multi-dimensional {Singular} {Integrals} and {Integral}
  {Equations}. ({Russian})}.
\newblock Moscow: Fizmatgiz, 1962.
\newblock 254 pages.

\bibitem{404a}
N.~Miller.
\newblock Weighted {S}obolev spaces and pseudodifferential operators with
  smooth symbols.
\newblock {\em Trans. Amer. Math. Soc.}, 269(1):91--109, 1982.



\bibitem{408}
B.~Muckenhoupt and R.L~Wheeden
\newblock Weighted norm inequalities for fractional integrals.
\newblock{\em Trans. Amer. Math. Soc.},  192: 261-274, 1974.

\bibitem{414b}
A.~Nekvinda.
\newblock Hardy-{Littlewood} maximal operator on ${L}^{p(x)}(\mathbb{R}^n)$.
\newblock {\em Math. Inequal. and Appl.}, 7(2):255--265, 2004.

\bibitem{479zz}
C.~P{\'e}rez.
\newblock Sharp estimates for commutators of singular integrals via iterations
  of the {H}ardy-{L}ittlewood maximal function.
\newblock {\em J. Fourier Anal. Appl.}, 3(6):743--756, 1997.

\bibitem{503a}
V.S. Rabinovich and S.G Samko.
\newblock Boundedness and {F}redholmness of pseudodifferential operators in
  variable exponent spaces.
\newblock {\em Integr. Eq. Oper. Theory},
 \newblock (to appear).

\bibitem{522b}
J.~L. Rubio~de Francia.
\newblock Factorization  and extrapolation of weights.
\newblock {\em Bull. Amer. J. Math. (N.S.)}, 7(2):393--395, 1982.

\bibitem{525}
M.~Ru$\check{z}$i$\check{c}$ka.
\newblock {\em Electroreological {Fluids}: {Modeling} and {Mathematical}
  {Theory}}.
\newblock Springer, {Lecture} {Notes} in {Math.}, 2000.
\newblock vol. 1748, 176 pages.

\bibitem{539}
N.G. Samko.
\newblock Singular integral operators in weighted spaces with generalized
  {H\"older} condition.
\newblock {\em Proc. {A}. {Razmadze} {Math}. {Inst}}, 120:107--134, 1999.

\bibitem{539e}
N.G. Samko.
\newblock On compactness of {I}ntegral {O}perators with a {G}eneralized {W}eak
  {S}ingularity in {W}eighted {S}paces of {C}ontinuous {F}unctions with a
  {G}iven {C}ontinuity {M}odulus.
\newblock {\em Proc. {A}. {Razmadze} {Math}. {Inst}}, 136:91, 2004.

\bibitem{539d}
N.G. Samko.
\newblock On non-equilibrated almost monotonic functions of the
  {Z}ygmund-{B}ary-{S}techkin class.
\newblock {\em Real Anal. Exch.}, 30(2):727--745, 2004/2005.



\bibitem{539j}
N.~Samko.
\newblock Parameter  depending  Bary-Stechkin classes and local  dimensions of measure metric
spaces.
\newblock {\em Proc. A.Razmadze Math. Inst.},   145 (2007), 122-129



\bibitem{539jnew}
N.~Samko.
\newblock Parameter depending almost monotonic functions and their applications to dimensions in metric
measure spaces.
\newblock {\em J. Funct. Spaces and  Appl.},    (2008),
\newblock to appear



\bibitem{539h}
N.Samko, S. Samko and B.Vakulov,
\newblock Weighted Sobolev
theorem in Lebesgue spaces with variable exponent,
\newblock {\em J. Math. Anal. and Applic.},
  \newblock 335(1): 560--583, 2007.

\bibitem{575a}
S.G. Samko.
\newblock Differentiation and integration of variable order and the spaces
  ${L}^{p(x)}$.
\newblock Proceed. of Intern. Conference "Operator Theory and Complex and
  Hypercomplex Analysis", 12--17 December 1994, Mexico City, Mexico, Contemp.
  Math., Vol. 212, 203-219, 1998.

\bibitem{579}
S.G. Samko.
\newblock Denseness of ${C_0^{\infty}({R}^N)}$ in the generalized {Sobolev}
  spaces $ {W^M,P(X)}{({R}^N)} $.
\newblock In {\em Intern. {Soc}. for {Analysis}, {Applic}. and {Comput}., vol.
  5, "Direct and Inverse {Problems} of {Math}. {Physics}", Ed. by R.Gilbert, J.
  Kajiwara and Yongzhi S. Xu, 333-342}. Kluwer {Acad}. {Publ}., 2000.

\bibitem{580b}
S.G. Samko.
\newblock Hardy inequality in the generalized {L}ebesgue spaces.
\newblock {\em Frac. Calc. and Appl. Anal}, 6(4): 355-362, 2003.


\bibitem{580bc}
S.G. Samko.
\newblock Hardy-{L}ittlewood-{S}tein-{W}eiss inequality in the {L}ebesgue
  spaces with variable exponent.
\newblock {\em Frac. Calc. and Appl. Anal}, 6(4):421--440, 2003.

\bibitem{580bd}
S.G. Samko.
\newblock On a progress in the theory of {L}ebesgue spaces with variable
  exponent: maximal and singular operators.
\newblock {\em Integr. Transf. and Spec. Funct}, 16(5-6):461--482, 2005.

\bibitem{584a}
S.G. Samko, E.~Shargorodsky, and B.~Vakulov.
\newblock Weighted {S}obolev theorem with variable exponent for spatial and
  spherical potential operators, {I}{I}.
\newblock {\em J. Math, Anal. Appl.}, 325(1):745--751, 2007.



\bibitem{730ab}
V.V.Zhikov.
\newblock On {L}avrentiev's phenomenon.
\newblock {\em Russian J. Math. Phys.}, 3(2):249--269, 1995.

\bibitem{730c}
V.V.Zhikov.
\newblock {M}eyer-type estimates for solving the non-linear {S}tokes system.
\newblock {\em Differ. Equat.}, 33(1): 108--115, 1997.



\end{thebibliography}

\def\ocirc#1{\ifmmode\setbox0=\hbox{$#1$}\dimen0=\ht0 \advance\dimen0
  by1pt\rlap{\hbox to\wd0{\hss\raise\dimen0
  \hbox{\hskip.2em$\scriptscriptstyle\circ$}\hss}}#1\else {\accent"17 #1}\fi}

\end{document}